\documentclass[a4paper, 11pt]{amsart}
\usepackage{enumerate}
\usepackage{color}
\usepackage{amssymb}
\usepackage{amsmath}

\numberwithin{equation}{section}
\usepackage{geometry}
\newtheorem{thm}{Theorem}[section]
\newtheorem{lem}[thm]{Lemma}
\newtheorem{prop}[thm]{Proposition}
\newtheorem{defi}[thm]{Definition}

\newtheorem{cor}[thm]{Corollary}
\newtheorem{que}[thm]{Question}
\newtheorem{rem}[thm]{Remark}
\newtheorem{fact}[thm]{Fact}

\newcommand\M{\mathcal{M}}
\newcommand\E{\mathcal{E}}
\newcommand{\h}{\mathsf{h}}
\newcommand\T{\tau}

\usepackage[colorlinks=true,urlcolor=blue, citecolor=blue,linkcolor=blue,linktocpage,pdfpagelabels,bookmarksnumbered,bookmarksopen]{hyperref}
\begin{document}

\title[Asymmetric Burkholder inequalities]{Asymmetric Burkholder inequalities in noncommutative symmetric spaces}


\author[L. Wu]{Lian Wu}
\address{School of Mathematics and Statistics, Central South University, Changsha 410083, People's Republic of China}
\email{wulian@csu.edu.cn}

\author[R. Xia]{Runlian Xia}
\address{School of Mathematics and Statistics, University of Glasgow, Glasgow, G12 8QQ, UK}
\email{Runlian.Xia@glasgow.ac.uk}

\author[D. Zhou]{Dejian Zhou}
\address{School of Mathematics and Statistics, Central South University, Changsha 410083, People's Republic of China}
\email{zhoudejian@csu.edu.cn}




\thanks{Lian Wu is supported by the NSFC (No.11971484). Lian Wu and Runlian Xia are supported by the Royal Society (IEC/NSFC/211199-International Exchanges 2021 Cost Share (NSFC)). Dejian Zhou is supported by the NSFC (No.12001541).}

\begin{abstract} In this paper, we establish noncommutative Burkholder inequalities with asymmetric diagonals in symmetric operator spaces. Our proof mainly relies on a noncommutative good-$\lambda$ approach and a new complex interpolation result on asymmetric vector valued spaces. We  include as well the asymmetric versions of noncommutative Johnson-Schechtman inequalities.
\end{abstract}

\maketitle

\section{Introduction}

Based on duality arguments, an asymmetric version of Doob's maximal inequalities for noncommutative martingales was established by Junge \cite{Ju}. This result answered a question posed by Pisier and initiated the study of asymmetric martingale inequalities in noncommutative setting. Via algebraic atomic decompositions and Davis decompositions, Hong et al. \cite{HJP2016} further discussed  asymmetric maximal inequalities in noncommutative $L_p$ spaces. Recently, their results  were extended to the context of noncommutative symmetric spaces in \cite{RWZ2021}.

Following \cite{Ju,HJP2016,RWZ2021},  the purpose of this paper is to discuss asymmetric Burkholder inequalities for noncommutative martingales in symmetric operator spaces.

To put our results in an appropriate context, let us  present some development of Burkholder's inequalities in the classical case. Inspired by searching for Banach spaces linearly isomorphic to a complemented subspace of an $L_p$ space, Rosenthal \cite{Rosenthal1970} established the following statement:  for $2\leq p<\infty$ and for every sequence of independent mean zero random variables $(g_n)_{n\geq 1}$ in $L_p$, we have
\begin{equation}\label{Rosenthal}
\Big\| \sum_{n\geq 1} g_n \Big\|_p^p \simeq_p \Big( \sum_{n\geq 1} \|g_n\|_2^2\Big)^{\frac 1 2}  + \Big( \sum_{n\geq 1} \|g_n\|_p^p \Big)^{1/p},
\end{equation}
where the notation $A\simeq_p B$ means that there exists a constant $c_p$ depending only on $p$  such that $c_p^{-1} B\le A\le c_p B$.  In 1973,  Burkholder \cite{Bu1} extended the equivalence \eqref{Rosenthal} to the  martingale setting.   It was shown in \cite{Bu1} that if $2\leq p<\infty$ and $(f_n)_{n\geq 1}$ is an $L_p$ bounded martingale (adapted to some discrete-time filtration $(\Sigma_n)_{n\geq 1}$),  then
\begin{equation}\label{Burkholder}
\|f\|_p \simeq_p \|s(f)\|_p + \Big(\sum_{n\geq 1}\|df_n\|_p^p\Big)^{1/p},
\end{equation}
where $df_n$ is the martingale difference of $f$ and  $s(f)$ is  the conditioned square function of  $f$. Moreover, as proved in \cite{Bu1}, the diagonal term  $(\sum_n\|df_n\|_p^p)^{1/p}$ can be replaced by the maximal function of martingale difference sequence; namely, if $2\leq p<\infty $ and if $f\in L_p,$ then
\begin{equation}\label{Burkholder-maximal}
\|f\|_{p}\simeq_p \|s(f)\|_{p}+\|\sup_n |df_n|\|_{p}.
\end{equation}
The above inequalities \eqref{Rosenthal}, \eqref{Burkholder} and \eqref{Burkholder-maximal} have turned out to be  very useful probalistic tools and have been found many interesting generalizations and applications; see for instance \cite{Astashikin2014, AS2005, AS3, CD1988, HOS,JQW2018,JS1989, KaltonSMS}.

In 1997, Pisier and Xu \cite{PX} formulated a noncommutative analogue of Burkholder-Gundy inequalities. This pioneering work stimulates the development of noncommutative martingale theory. A lot of  classical martingale inequalities have been gradually generalized to the noncommutative setting; see \cite{BCO, CRX, Dirksen2,Ju,JM2007,JX2003,JX2008, Ran2002,Ran2005,Ran2007, Ra2023} and many others. In particular, the work due to Junge and Xu \cite{JX2003,JX2008} provided a comprehensive study of \eqref{Rosenthal}, \eqref{Burkholder} and \eqref{Burkholder-maximal} in the context of noncommutative $L_p$ spaces. One of the main results in \cite{JX2003} can be summarized as follows: if $2\leq p<\infty$ and if $x\in L_p(\mathcal{M}),$ then
\begin{equation}\label{nc1}
\big\|x\big\|_p \simeq_p \max\Big\{ \big\|s_c(x)\big\|_p,  \big\|s_r(x)\big\|_p, \big( \sum_{n\geq 1} \big\|dx_n\big\|_p^p \big)^{1/p}\Big\},
\end{equation}
where $s_c(x)$ and $s_r(x)$ denote the column and the row versions of  conditioned square functions for which we refer to the next section for concrete definitions. Here and what follows, $(\M,\tau)$ will always denote a finite von Neumann algebra equipped with a faithful normal finite trace $\T$ with $\T({\bf 1})=1$.  Moreover, by using an interpolation method, Junge and Xu \cite{JX2008} proved that, similar to the classical case, the diagonal term in \eqref{nc1} can be replaced by a maximal function version: if $2< p<\infty$ and if $x\in L_p(\mathcal{M}),$ then
\begin{equation}\label{nc2}
\big\|x\big\|_p \simeq_p \max\Big\{ \big\|s_c(x)\big\|_p,  \big\|s_r(x)\big\|_p,  \|(dx_n)_{n\geq 1}\big\|_{L_p(\mathcal M; \ell_\infty)}\Big\}.
\end{equation}
The above results \eqref{nc1}, \eqref{nc2} can be regarded as noncommutative extensions of \eqref{Burkholder}, \eqref{Burkholder-maximal}, respectively.

Recently, the equivalence \eqref{nc1} was investigated in more general symmetric operator spaces by some authors  (c.f. \cite{CR2021, Dirksen2,RW,Ran-Wu-Xu}). In particular, the following is proved in \cite{Ran-Wu-Xu}: if the symmetric Banach function space
$E$ with Fatou property lies in $ \textrm{Int}(L_2, L_q)$ for some $2< q < \infty$ and if $x\in E(\mathcal M)$, then
\begin{equation}\label{ncE}
\big\|x\big\|_E \simeq_E \max\Big\{ \big\|s_c(x)\big\|_E,  \big\|s_r(x)\big\|_E, \Big\|\sum_{n\geq1}dx_n\otimes e_n\Big\|_{E(\mathcal M\overline\otimes\ell_\infty)}\Big\},
\end{equation}
where $(e_n)_{n\geq 1} $ are the standard unit vectors in $\ell_\infty$.
Applying a noncommutative good-$\lambda$ approach (we refer to \cite{JOW2} for more information about this method), Jiao, Zanin and Zhou \cite{JZZ} generalized \eqref{nc2} to the symmetric operator spaces: if the symmetric Banach function space
$E$ lies in  $\textrm{Int}(L_p, L_q)$ for $2< p\leq q < \infty$ and if $x\in E(\mathcal M)$, then
\begin{equation}\label{ncE-maximal}
\big\|x\big\|_E \simeq_E \max\Big\{ \big\|s_c(x)\big\|_E,  \big\|s_r(x)\big\|_E, \|(dx_n)_{n\geq 1}\big\|_{E(\mathcal M; \ell_\infty)}\Big\},
\end{equation}

The purpose of this paper is to explore asymmetric versions of \eqref{ncE-maximal},  in the sense that the maximal diagonal part is replaced with an ``asymmetric" maximal diagonal term. As mentioned in the beginning, this is motivated by recent progress on asymmetric maximal inequalities of noncommutative martingales (\cite{Ju,HJP2016,RWZ2021}).

Our first version of asymmetric Burkholder inequalities reads as follows.
We should mention that a special case (i.e., $\theta$=1 or $0$) of the following result was achieved in \cite{RWZ2021} by using Davis decompositions.

\begin{thm}\label{discrete-no-convex}
Let $0\leq \theta\leq 1$. Assume that $E$ is a symmetric Banach function space  which is an interpolation of the couple $ (L_p,L_q)$ for  $2<p\leq q<\infty$. If $x\in E(\mathcal M)$, then
\begin{equation*}
\big\|x\big\|_E \simeq_E \max\Big\{ \big\|s_c(x)\big\|_E,  \big\|s_r(x)\big\|_E, \|(dx_n)_{n\geq 1}\big\|_{E(\mathcal M; \ell_\infty^\theta)},  \|(dx_n)_{n\geq 1}\big\|_{E(\mathcal M; \ell_\infty^{1-\theta})}\Big\}.
\end{equation*}
\end{thm}

We might also prove a stronger version of asymmetric Burkholder inequalities under a more strict assumption on the symmetric function space.

\begin{thm}\label{discrete-version-E}
Let $0\leq \theta\leq 1$. Assume that $E$ is a symmetric Banach function space  which is an interpolation of the couple $ (L_p,L_q)$ for  $2<p\leq q<\infty$.  If $E$ has Fatou norm and is $2$-convex, then for $x\in E(\mathcal M)$, we have
\begin{equation*}\label{nc-asymmetric-1}
	\big\|x\big\|_E \simeq_E \max\Big\{ \big\|s_c(x)\big\|_E,  \big\|s_r(x)\big\|_E, \|(dx_n)_{n\geq 1}\big\|_{E(\mathcal M; \ell_\infty^\theta)}\Big\}.
\end{equation*}
\end{thm}

Once Theorem \ref{discrete-no-convex} and Theorem \ref{discrete-version-E} have been established, a natural question is to consider the  dual versions. Our third result deals with this problem.

\begin{thm}\label{duality-ns}
Let $0\leq \theta\leq 1$. Assume that $E$ is a symmetric Banach function space  which is an interpolation of the couple $ (L_p,L_q)$ for  $1<p\leq q<2$. If  $E$ is $2$-concave and $x\in E(\mathcal M)$, then
\begin{equation*}\label{nc-asymmetric-2}
		\big\|x\big\|_E \simeq_E \inf\Big\{ \big\|s_c(y)\big\|_E+\big\|s_r(z)\big\|_E+\|(dw_n)_{n\geq 1}\big\|_{E(\mathcal M; \ell_1^\theta)}\Big\},
	\end{equation*}
where the infimum is taken over all decompositions $x=y+z+w$ with $y$, $z$ and $w$ being martingales.
\end{thm}

As one may notice, Theorem \ref{discrete-no-convex}, Theorem \ref{discrete-version-E}, and Theorem \ref{duality-ns} are new even for $E=L_p$.  They strengthen the results from \cite{Dirksen2,JZZ,RW,Ran-Wu-Xu}, and in particular, Theorem \ref{discrete-no-convex}  improves \cite[Corollary 4.11]{RWZ2021} greatly.

The proof of Theorem \ref{discrete-no-convex} is mainly based on   a noncommutative good-$\lambda$ approach. Our strategy for the proof of Theorem  \ref{discrete-version-E} and Theorem \ref{duality-ns} are as follows. To prove Theorem \ref{discrete-version-E}, we establish a new complex interpolation result for asymmetric vector valued spaces. Then, using a duality approach, we may deduce Theorem \ref{duality-ns} from Theorem \ref{discrete-version-E}.

Our arguments for Theorem \ref{discrete-version-E} can be also applied to get an asymmetric version of  noncommutative Johnson-Schechtman inequalities. Recall that Johnson and Schechtman \cite{JS1989} extended the Rosenthal inequalities \eqref{Rosenthal} to the symmetric  Banach function spaces. This result has been extensively investigated in the noncommutative setting (see \cite{Sukochev2012,JSZ2016,JSXZ,JSZZ}). Recently, a version of noncommutative Johnson-Schechtman inequalities involving maximal diagonal term was given in \cite[Theorem 1.5]{JZZ}. In this paper, by exploiting our new complex interpolation, we further extend their result to the asymmetric case (see Theorem \ref{main-result-2} below).

The paper is organized as follows. In the next
preliminary section, we review the basics of noncommutative spaces and noncommutative martingales. This section also includes some background on the pointwise product of functions spaces and complex interpolation method. Section \ref{asymmetric-space} is devoted to studying several fundamental properties of the asymmetric vector valued spaces $E(\mathcal M;\ell_\infty^\theta)$ and $E(\mathcal M;\ell_1^\theta)$.   In Section \ref{complex-interpolation}, we will establish a complex interpolation result for $E(\mathcal M;\ell_\infty^\theta)$, which constitutes one of the key ingredients in the proof of asymmetric inequalities. This result, together with a complementary argument, yields the corresponding interpolation result for the martingale Hardy spaces $\h_E^{\infty_\theta}$.  Section \ref{sec-asymmetric-burkholer} contains the proofs of asymmetric Burkholder inequalities (namely, Theorem \ref{discrete-no-convex} and Theorem \ref{discrete-version-E}).
A dual theorem between $E(\mathcal M;\ell_\infty^\theta)$ and $E(\mathcal M;\ell_1^\theta)$  is established in Section \ref{main-duality}. Combining the duality with a complementary result, we   also prove a similar duality for the martingale Hardy spaces $\h_E^{\infty_\theta}$ and $\h_E^{1_\theta}$. Using this duality, we   deduce Theorem \ref{duality-ns} from Theorem \ref{discrete-version-E}. The last section focus on asymmetric form of Johnson-Schechtman inequalities.

We close the introduction with an interesting question.

\begin{que} Let $0\leq \theta\leq 1$.  Assume that $E$ is a symmetric Banach function space satisfies $E\in \emph{Int}(L_p,L_q)$ with $\max\{2(1-\theta),2\theta\}< p\leq q<\infty$. For $x\in E(\mathcal M)$, do we have
\begin{equation*}
\inf\Big\{ \big\|s_c(y)\big\|_E+\big\|s_r(z)\big\|_E+\|(dw_n)_{n\geq 1}\big\|_{E(\mathcal M; \ell_1^\theta)}\Big\}\lesssim_E \big\|x\big\|_E,
\end{equation*}
where the infimum is taken over all decompositions $x=y+z+w$ with $y$, $z$, $w$ being martingales, and
\begin{equation*}
\big\|x\big\|_E \lesssim_E \max\Big\{ \big\|s_c(x)\big\|_E,  \big\|s_r(x)\big\|_E, \|(dx_n)_{n\geq 1}\big\|_{E(\mathcal M; \ell_\infty^\theta)}\Big\}?
\end{equation*}
\end{que}


\section{Preliminaries}\label{prelims section}

Throughout this paper, we write $A\lesssim_\alpha B$ if there is a constant $C_\alpha$ depending only on the parameter $\alpha$ such that  the inequality $A \leq C_\alpha B$ is satisfied, and write $A\simeq_\alpha B$ if both $A\lesssim_\alpha B$ and $B \lesssim_\alpha A$ hold.

\subsection{Noncommutative symmetric spaces}
Throughout, $(\M,\tau)$ will always denote a finite von Neumann algebra equipped with a faithful normal finite trace $\T$. Without loss of generality, we assume that $\tau({\bf 1})=1$. Let  $L_0(\M,\T)$ be  the associated topological $*$-algebra of measurable operators  in the sense of \cite{Nel74}; see also \cite{Terp1982} for more details. Since $\T$ is finite, $L_0(\M,\T)$ consists of all the operators affiliated to $\M$. For $x\in L_0(\M,\T)$, we define  its generalized singular number $\mu(x)$ by
$$\mu_t(x)=\inf\{s>0: \tau(\chi_{(s,\infty)}(|x|))\leq t\},\quad 0<t\leq 1,$$
where $\chi_{(s,\infty)}(|x|)$ is the spectral projection of $|x|$. If $\M$ is the abelian von Neumann algebra $L_\infty(0,1]$ with  the trace given by  integration  with respect to Lebesgue  measure, then $L_0(\M,\T)$ is  the space of measurable  complex valued functions  on $(0,1]$. For $f\in L_0(\M,\T)$,  $\mu(f)$ is the usual decreasing rearrangement of  $|f|$. We refer to \cite{PX3} for more information on noncommutative integration.

A Banach (or, quasi Banach) function space  $(E,\|\cdot\|_E)$ of measurable functions  on the interval $(0,1]$ is called symmetric if for every $g \in E$ and every measurable function $f$ with $\mu(f) \leq \mu(g)$, we have $f \in E$ and $\|f\|_E \leq \|g\|_E$. Moreover, $E$ is called fully symmetric if for every $g \in E$ and every measurable function $f$ with $\mu(f) \prec\prec \mu(g)$ (which means $\int_0^t \mu(s,f)ds\leq \int_0^t \mu(s,g)ds$ for $0<t\leq 1$), we have $f \in E$ and $\|f\|_E \leq \|g\|_E$.

A symmetric Banach function space $E$ is said to have Fatou norm if for every net $(f_\beta)\subset E$ and  $f\in E$ satisfying $0\leq f_{\beta} \uparrow f$, we have  $\|f_{\beta}\|_{E} \uparrow \|f\|_E$.
	A  symmetric Banach space  $E$ is said to have Fatou property if for every net $(f_\beta)\subset E$ and measurable function $f$ satisfying $0\leq f_{\beta} \uparrow f$ and $\sup_{\beta}\|f_{\beta}\|_E<\infty$, we have $f\in E$ and $\|f_{\beta}\|_{E} \uparrow \|f\|_E$. We say that $E$ has
order continuous norm if for every net $(f_\beta)$ in $E$ such that $f_\beta \downarrow 0$ we have $\|f_\beta\|_E \downarrow 0$.

For a symmetric Banach space  $E$, we define the  corresponding  noncommutative space by
\begin{equation}\label{E-def}
E(\M, \T) = \Big\{ x \in
L_0(\M,\T)\ : \ \mu(x) \in E \Big\}.
\end{equation}
Endowed with the norm
$\|x\|_{E(\M,\T)} := \| \mu(x)\|_E$,   the linear space $E(\M,\T)$ becomes a Banach space (\cite{Kalton-Sukochev,X}) and is usually referred to as the noncommutative symmetric space associated with $(\M,\T)$ and  $E$. An extensive discussion of the various properties of such spaces can be found in \cite{DDP3,PX3,X}.
 We remark that  in particular if  $E=L_p(0, 1]$ with $1\leq p<\infty$, then $E(\M, \T)=L_p(\M,\T)$  where $L_p(\M,\T)$ is  the noncommutative $L_p$-space associated with  $(\M,\T)$. In the sequel, $E(\M,\T)$ will be abbreviated to $E(\M)$ and the norm will be denoted by $\|\cdot\|_E$;  $E(\M)^+$ is denoted by the set of all positive elements in $E(\M)$.

In this paper, we  consider  symmetric spaces  that are interpolations of the couple $(L_p, L_q)$ for $1\leq p<q\leq \infty$. For a given compatible  Banach couple  $(X, Y)$, we recall that a Banach space $Z$ is called  an interpolation space if $X \cap Y \subseteq Z \subseteq X +Y$ and whenever  a bounded linear operator $T: X + Y \to X + Y$ is such that $T(X) \subseteq X$ and $T(Y) \subseteq Y$, we have $T(Z)\subseteq Z$
and $\|T:Z\to Z\|\leq C\max\{\|T :X \to X\| ,\; \|T: Y \to Y\|\}$ for some constant $C$.
In this case, we write $Z\in {\rm Int}(X,Y)$.

\begin{lem}[{\cite[Theorem II.3.4]{KPS1982}}]\label{fully-symmetric-lem}
A symmetric quasi Banach function space $E$ lies in $ \emph{Int}(L_1,L_\infty)$ if and only if $E$ is fully symmetric.
\end{lem}

Given a symmetric space $E$, the  K\"othe dual of  a symmetric space $E$ is the function space defined by setting
 \[
 E^\times=\left\{ f \in L_0(0,1]: \int_0^1 |f(t)g(t)|\  dt <\infty, \, \forall g \in E\right\};
 \]
$$\|f\|_{E^\times}:=\sup\left\{\int_0^1|f(t)g(t)|\ dt : \|g\|_E \leq 1\right\},\quad f\in E^\times.$$
The space $E^\times$ is fully symmetric and has the Fatou property.

A symmetric Banach function space $E$ on $(0, 1]$ has a Fatou norm if and only if
$E$ embeds isometrically into its second K\"{o}the dual $E^{\times \times}$. It has the Fatou property if and only if $E=E^{\times \times}$ isometrically. It has order continuous
norm if and only if it is separable, which is also equivalent to the statement
$E^* = E^\times$.  It is well-known that if $E\in {\rm Int}(L_p, L_q)$ for some $1\leq p<q\leq \infty$, then $E^\times \in {\rm Int}(L_{q'}, L_{p'})$ where $p'$ and $q'$ denote the conjugate indices of $p$ and $q$ respectively.

For $0<r<\infty$,  the $r$-convexification  of a Banach function space $E$ is defined by
\begin{equation*}
E^{(r)} :=\left\{ f \in L_0(0,1] : |f|^r \in E \right\}
\end{equation*}
equipped with the norm (or, quasi-norm)
 \begin{equation*}\big\|f \big\|_{E^{(r)}} = \|| f|^r\|_E^{\frac{1}{r}}.
\end{equation*}
 It is easy to verify that if $E$ is symmetric then so is  $E^{(r)}$.  Moreover, according to \cite[Proposition~3.5]{Dirk-Pag-Pot-Suk},  if  $E\in {\rm  Int}(L_p,L_q)$ for some $1\leq p < q \leq \infty$, then we have
\begin{equation}\label{eq: int}
  E^{(r )} \in {\rm Int}(L_{pr}, L_{qr}).
\end{equation}
 These facts will be used repeatedly throughout.

 Let $0<p, q<\infty$.  A symmetric quasi Banach function space $E$ is said to be $p$-convex if there is a constant $c>0$ such that for any finite sequence $(f_i)_{i=1}^n$ in $E$ we have
 	$$\left\|\left(\sum_{i=1}^n|f_i|^p\right)^{1/p}\right\|_E\leq c \left(\sum_{i=1}^n \|f_i\|_E^p\right)^{1/p}.$$
The least constant satisfying the above inequality is called the $p$-convexity constant of $E$ and will be denoted by $M^{(p)}(E)$.  A symmetric quasi Banach function space $E$ is said to be $q$-concave if there is a constant $c>0$ such that for any finite sequence $(f_i)_{i=1}^n$ in $E$ we have
 	$$\left(\sum_{i=1}^n \|f_i\|_E^q\right)^{1/q} \leq c \left\|\left(\sum_{i=1}^n|f_i|^q\right)^{1/q}\right\|_E.$$
 The least constant satisfying the above inequality is called the $q$-concavity constant of $E$ and will be denoted by $M_{(q)}(E)$.
From
the definitions, one can easily show that if $E$ is $p$-convex and $q$-concave for $0 <
p \leq  q < \infty$, then $E^{(r)}$
is $pr$-convex and $qr$-concave. Moreover, if $E$ is $p$-convex for some $0 < p < \infty$, then $E$ is $r$-convex for any $0 < r \leq p $ and $M^{(r)} \leq M^{(p)}$. Similarly, if $E$ is $q$-concave for some $0 < q <\infty$, then $E$ is $s$-concave for any $q \leq s \leq \infty$
and $M_{(s)}\leq M_{(q)}$. It is also well known that if $E$ is $q$-concave for some $0<q<\infty$, then $E$ has order continuous norm, which is equivalent to saying that $E$ is separable (see \cite[Lemma 4.12]{Dirksen-Thesis}).

 The following result is taken from \cite[Proposition 3.4]{DDS2014}.

\begin{prop}\label{DDS2014}
	If $E$ is a symmetric Banach function space with Fatou norm, and if $E$ is $p$-convex for some $0<p<\infty$, then there exits a symmetric norm  $\|\cdot\|_0$ on $E^{(1/p)}$ such that
	$$\|x\|_{0}\leq \|x\|_{E^{(1/p)}}\leq M^{(p)}(E)^p\|x\|_{0}.$$
\end{prop}

We now discuss the pointwise  product of function spaces that we will need in the sequel. Let $E_i$ ($i=1,2$) be symmetric Banach function spaces on $(0,1]$. We define the pointwise product  of $E_1$ and $E_2$ by setting:
$$E_1 \odot E_2 =\big\{x:x=x_1x_2, x_1\in E_1, x_2\in E_2\big\}.$$
For $f\in E_1\odot E_2$, set
$$\|f\|_{E_1 \odot E_2} :=\inf\left\{\|f_1\|_{E_1}\|f_2\|_{E_2}:f_1\in E_1, f_2\in E_2, f=f_1f_2\right\}.$$
According to \cite[Theorem 2]{KLM}, $E_1 \odot E_2$  is a symmetric quasi Banach function space  when equipped with the quasi norm $\|\cdot\|_{E_1 \odot E_2}$. Similarly, one can also define the product of symmetric operator spaces as follows:
$$E_1(\mathcal M) \odot E_2(\mathcal M) =\big\{x:x=x_1x_2, x_1\in E_1(\mathcal M), x_2\in E_2(\mathcal M)\big\}.$$
For $x\in E_1(\mathcal M) \odot E_2(\mathcal M)$, we define
$$\|x\|_{E_1(\mathcal M) \odot E_2(\mathcal M)} :=\inf\big\{\|x_1\|_{E_1}\|x_2\|_{E_2}:x_1\in E_1(\mathcal M), x_2\in E_2(\mathcal M), x=x_1x_2\big\}.$$
From \cite[Theorems 3, 4]{Sukochev2016} (see also \cite[Theorem 2.5]{Bekjan2015}), we know that if $E=E_1\odot E_2$, then
$E(\mathcal M)= E_1(\mathcal M)\odot E_2(\mathcal M).$

We record some facts about pointwise products of symmetric spaces for further use.

\begin{lem}[{\cite[Corollary 2, Theorem 1]{KLM}}]\label{prod-1}
Let $X$, $Y$ be  symmetric Banach function spaces. Then the following hold:
\begin{enumerate}[\rm (i)]
  \item if $1<p<\infty$, then $[X^{(p)}]^\times=[X^\times]^{(p)}\odot L_{p'},$ where $p'$ is the conjugate index of $p$;
  \item if $0<p<\infty$, then $(X\odot Y)^{(p)}=X^{(p)}\odot Y^{(p)}$.
\end{enumerate}
\end{lem}


\begin{lem}[{\cite[Lemma 3.1]{Bekjan2018}}]\label{useful-factorization}
Let $E$, $E_1$ and $E_2$ be symmetric Banach function spaces on $(0,1]$ such that $E=E_1\odot E_2$. If $x\in E(\mathcal M)^+$, then for any $\varepsilon>0$, there exist $a\in E_1(\mathcal M)^+$ and $b\in E_2(\mathcal M)^+$ such that $x=ab$,
$$\|a\|_{E_1} \|b\|_{E_2}\leq (1+\varepsilon)\|x\|_{E}$$
and $a$ is invertible with bounded inverse.
\end{lem}

%
%

\subsection{Martingales and Hardy spaces}

We now  describe the general setup for  martingales in noncommutative symmetric spaces.
  Denote by $(\M_n)_{n \geq 1}$ an
increasing sequence of von Neumann subalgebras of ${\M}$
whose union  is weak*-dense in
$\M$. For $n\geq 1$, we assume that there exists a trace preserving conditional expectation ${\E}_n$
from ${\M}$ onto  ${\M}_n$.  It is well-known that if  $\T_n$  denotes the restriction of $\T$ on $\M_n$, then $\E_n$ extends to a contractive projection from $L_p(\M,\T)$ onto $L_p(\M_n, \T_n)$ for all $1\leq p \leq \infty$. More generally, if $E\in {\rm Int}
(L_1, L_\infty)$, then  for every $n\geq 1$,   $\E_n$ is bounded  from $E(\M,\T)$ onto $E(\M_n,\T_n)$.

\begin{defi}
A sequence $x = (x_n)_{n\geq 1}$ in $L_1(\M)$ is called a
noncommutative martingale with respect to $({\M}_n)_{n \geq
1}$ if $\mathcal{E}_n (x_{n+1}) = x_n$ for every $n \geq 1.$
\end{defi}
If, in addition, all $x_n$'s belong to $E(\M)$,  then  $x$ is called an $E(\M)$-martingale.
In this case, we set
\begin{equation*}\| x \|_{E} = \sup_{n \geq 1} \|
x_n \|_{E}.
\end{equation*}
If $\| x \|_{E} < \infty$, then $x$ is called
a bounded $E(\M)$-martingale.

Let $x = (x_n)_{n\geq1}$  be a noncommutative martingale with respect to
$(\M_n)_{n \geq 1}$.  Define $dx_n = x_n - x_{n-1}$ for $n
\geq 2$ and $d x_1=x_1$. The sequence $dx =
(dx_n)_{n\geq 1}$ is called the martingale difference sequence of $x$. A martingale
$x$ is called  a finite martingale if there exists $N$ such
that $dx_n = 0$ for all $n \geq N.$
In the sequel, for any operator $x\in L_1(\M)$,  we denote $x_n=\E_n(x)$ for $n\geq 1$.

Let us now  review the definitions of the conditioned square functions and  martingale conditioned Hardy spaces for the case of  noncommutative  symmetric spaces.
%
Let $x = (x_n)_{n \geq 1}$ be a martingale in $L_2(\M)$.
We set (with the convention that  $\E_0=\E_1$):
 \[
 s_{c,n} (x) = \Big ( \sum^n_{k = 1} \E_{k-1}|dx_k |^2 \Big )^{\frac 1 2}, \quad
 s_c (x) = \Big ( \sum^{\infty}_{k = 1} \E_{k-1}|dx_k|^2 \Big )^{\frac 1 2}\,.
 \]
 The operator $s_c(x)$ is  called  the column conditioned square function of $x$. In order to define the  conditioned Hardy spaces, we need to use Junge's  representation of conditioned spaces  (\cite{Ju}). The corresponding extension to noncommutative  symmetric spaces can be found in \cite{RW,RW2,Ran-Wu-Xu}, but since we will need some of the finer details in our proofs, we include a brief description.

 For every $n\geq 1$ and $1\leq p \leq \infty$,   define the  space $L_p^c(\M,\E_n)$ to be the completion of $\M$ with respect to the quasi-norm
 \[
 \big\|x\big\|_{L_p^c(\M,\E_n)} =\big\|\E_n(x^*x) \big\|_{p/2}^{\frac 1 2}.
 \]
 According to \cite[Proposition 2.8]{Ju}, there exists an isometric right $\M_n$-module map
  $$u_{n,p}: L_p^c(\M, \E_n) \to L_p(\M_n;\ell_2^c)$$
   such that
\begin{equation}\label{u}
u_{n,p}(x)^* u_{n,q}(y)=\E_n(x^*y) \otimes e_{1,1}
\end{equation}
for all $x\in L_p^c(\M,\E_n)$ and $y \in L_q^c(\M,\E_n)$ with $1/p +1/q \leq 1$.

Denote by $\mathcal{F}$ the collection of all finite sequences $(a_n)_{n\geq 1}$ in $\M$.
 For $1\leq p\leq \infty$,  define  the space $L_p^{\rm cond}(\M; \ell_2^c)$ to be   the completion  of $\mathcal{F}$  with respect to the norm
\begin{equation*}
\big\| (a_n)_{n\geq 1} \big\|_{L_p^{\rm cond}(\M; \ell_2^c)} = \Big\| \Big(\sum_{n\geq 1} \E_{n-1}|a_n|^2\Big)^{\frac 1 2}\Big\|_p.
\end{equation*}
The space $L_p^{\rm cond}(\M;\ell_2^c)$ can be isometrically embedded into an $L_p$-space associated to  a semifinite von  Neumann algebra  by means of the following map:
\[
U_p : L_p^{\rm cond}(\M; \ell_2^c) \to L_p(\M \overline{\otimes} B(\ell_2(\mathbb{N}^2)))\]
defined by setting
\[
U_p((a_n)_{n\geq 1}) = \sum_{n\geq 1} u_{n-1,p}(a_n) \otimes e_{n,1}.
\]
From \eqref{u}, it follows
   that if $(a_n)_{n\geq 1} \in L_p^{\rm cond}(\M;  \ell_2^c)$ and $(b_n)_{n\geq 1} \in L_q^{\rm cond}(\M;  \ell_2^c)$  for $1/p +1/q \leq 1$, then
\begin{equation*}
U_p( (a_n)_{n\geq 1})^* U_q((b_n)_{n\geq 1}) =\Big(\sum_{n\geq 1} \E_{n-1}(a_n^* b_n)\Big) \otimes e_{1,1} \otimes e_{1,1}.
\end{equation*}
In particular,  $\| (a_n)_{n\geq 1} \|_{L_p^{\rm cond}(\M; \ell_2^c)}=\|U_p( (a_n)_{n\geq 1})\|_p$ and
hence  $U_p$ is indeed an isometry.  An important fact here is that  $U_p$ is independent of $p$ and hence  we will simply write $U$ for $U_p$.

 Now, we  generalize  the notion of conditioned spaces to the setting of symmetric spaces.  Fix a symmetric Banach function space $E$. We consider  the  algebraic linear  map $U$ restricted to   the linear space $\mathcal{F}$  that takes  its values in $L_1(\M \overline{\otimes} B(\ell_2(\mathbb{N}^2))) \cap \big(\M \overline{\otimes} B(\ell_2(\mathbb{N}^2))\big)$.  For a given
   sequence $(a_n)_{n\geq 1} \in \mathcal{F}$, we set:

\begin{align*}
\big\| (a_n)_{n\geq 1} \big\|_{E^{\rm cond}(\M; \ell_2^c)}& = \Big\| \big(\sum_{n\geq 1} \E_{n-1}|a_n|^2\big)^{\frac 1 2}\Big\|_{E(\M)}\\
& =\big\| U( (a_n)_{n\geq 1})\big\|_{E(\M \overline{\otimes} B(\ell_2(\mathbb{N}^2)))}.
\end{align*}

This is well-defined and  induces   a norm  on the linear space $\mathcal{F}$.  We define   the Banach space $E^{\rm cond}(\M;  \ell_2^c)$ to be the completion of $\mathcal{F}$  with respect to the above  norm. Then
 $U$ extends to an isometry from  $E^{\rm cond}(\M;\ell_2^c)$ into $E(\M \overline{\otimes} B(\ell_2(\mathbb{N}^2)))$ which we will still denote by $U$.

 Let $x=(x_n)_{n\geq 1}$ be a finite martingale in $ E(\M)$.  Define
 \[
 \big\|x\big\|_{\h_E^c} =\big\| (dx_n)\big\|_{E^{\rm cond}(\M; \ell_2^c)}=\big\|s_c(x)\big\|_{E(\M)}.
 \]
 From the definition of $E^{\rm cond}(\M; \ell_2^c)$,  one can easily see that $\|\cdot\|_{\h_E^c}$  is a norm on the set of all finite martingales in $\M$. The conditioned Hardy space $\h_E^c(\M)$ is defined as the completion of the set all finite martingales under the norm $\|\cdot\|_{\h_E^c}$.

 If we denote by $\mathcal{D}_c : \h_E^c(\M) \to E^{\rm cond}(\M;\ell_2^c)$ the \lq\lq extension\rq\rq of the  natural map $x \mapsto (dx_n)_{n\geq 1}$, then its composition with $U$
 induces  the isometric embedding:
\[
U\mathcal{D}_c: \h_E^c(\M) \to E(\M \overline{\otimes} B(\ell_2(\mathbb{N}^2)))
\]
with the property that  if $x \in \h_E^c(\M)$ and $y\in \h_{E^\times}^c(\M)$ then
\begin{equation*}
\big(U\mathcal{D}_c(x)\big)^* U\mathcal{D}_c(y)=\Big( \sum_{n\geq 1}\E_{n-1}(dx_n^*dy_n)\Big) \otimes e_{1,1} \otimes e_{1,1}.
\end{equation*}
In particular, whenever $s_c(x)$ is  a well-defined operator, we have
\begin{equation*}
|U\mathcal{D}_c(x)|^2 =(s_c(x))^2 \otimes e_{1,1} \otimes e_{1,1}.
\end{equation*}




Similarly, we may define the  row conditioned square functions $s_r(x)$ and row conditioned Hardy spaces $\h_E^r(\M)$.

\subsection{Complex interpolation}\label{Complex}

Let $(X_0, X_1)$ be an interpolation couple of  Banach spaces; namely $X_j$, $j=0, 1$
are continuously embedded into a Hausdorff topological vector space $Y$. Let $\mathcal S$ (respectively, $\overline{\mathcal S}$) denote the open strip $\{z: 0<\textrm{Re}z<1\}$
(respectively, the closed strip $\{z : 0 \leq \textrm{Re}z\leq 1\}$) in the complex plane $\mathbb C$. Denote by $\partial_0=\{z\in \overline{\mathcal S}: \textrm{Re}z=0\}$, $\partial_1=\{z\in \overline{\mathcal S}: \textrm{Re}z=1\}$  the boundaries of $\overline{\mathcal S}$. Let $\mathcal F(X_0, X_1)$ be the space of bounded analytic functions $f : \mathcal S \to X_0 + X_1$ which
extend continuously to $\overline{\mathcal S}$  such that the functions $t \mapsto f (j + it)$ are bounded and continuous from $\mathbb{R}$ into $X_j$, $j = 0, 1$. We equip $\mathcal F(X_0, X_1)$ with the norm
$$\|f\|_{\mathcal F(X_0, X_1)}=\max\Big\{\sup_{z\in \partial_0}\|f(z)\|_{X_0}, \sup_{z\in \partial_1}\|f(z)\|_{X_1}\Big\}.$$
Then $\mathcal F(X_0, X_1)$ is a  Banach space. For $0\leq\theta\leq 1$, we define the complex interpolation space $[X_0, X_1]_\theta$  as the set of all  $x\in  X_0+X_1$ satisfying that $x=f(\theta)$ for some $f\in \mathcal F(X_0, X_1)$. The norm on $[X_0, X_1]_\theta$ is defined by setting
$$\|x\|_{[X_0, X_1]_\theta}=\inf\left\{\|f\|_{\mathcal F(X_0, X_1)}:f\in \mathcal F(X_0, X_1), f(\theta)=x\right\}.$$
It then follows that $[X_0, X_1]_\theta$ is a  Banach space for $0\leq \theta \leq 1$ (see \cite[Theorem 4.1.2]{BL1976}).

The following complex interpolation result might be well known to experts. We still include a short proof for the convenience of the reader.

\begin{lem}\label{complex-E-infty} Let $0<\theta<1$.  Assume that $E_1$, $E_2$ are fully symmetric Banach function spaces. We have
$$[E_1(\mathcal M), E_2(\mathcal M)]_\theta=E_1^{(\frac{1}{1-\theta})}(\mathcal M)\odot E_2^{(\frac{1}{\theta})}(\mathcal M).$$
\end{lem}
\begin{proof}
The desired result follows from \cite[Theorem 4.6]{KaltonSMS} and \cite[Theorem 3.2]{DDP1992}. In fact, by \cite[Theorem 4.6]{KaltonSMS}, we have
$$[E_1,E_2]_\theta=E_1^{(\frac{1}{1-\theta})}\odot E_2^{(\frac{1}{\theta})}.$$
Now, applying \cite[Theorem 3.2]{DDP1992}, the preceding interpolation automatically lifts to the noncommutative setting.
\end{proof}

\section{Asymmetric vector valued spaces}\label{asymmetric-space}
In this section, we introduce two asymmetric vector valued operator spaces associated with symmetric Banach function spaces $E$: $E(\mathcal M;\ell_\infty^\theta)$ and $E(\mathcal M;\ell_1^\theta)$. Some elementary results about these spaces are presented for further use.

Let us begin with the definition of $E(\mathcal M;\ell_\infty^\theta)$. Let $0\leq \theta\leq 1$. Suppose that $E$ is a symmetric Banach function space such that $E\in \textrm{Int}(L_p, L_q)$ with $1\leq p\leq q\leq \infty$. We define $E(\mathcal M; \ell_\infty^\theta)$ to be the space of all sequences $x=(x_n)_{n\geq 1}$  in $E(\mathcal M)$ for which there exist $a\in E^{(\frac{1}{1-\theta})}(\mathcal M),\,\, b\in E^{(\frac{1}{\theta})}(\mathcal M)$ and $y=(y_n)_{n\geq1}\subset L_\infty(\mathcal M)$
such that
\begin{equation}\label{linfty-theta-factorization}
x_n=ay_n b,\quad n\geq 1.
\end{equation}
For $x\in E(\mathcal M; \ell_\infty^\theta)$, we define
$$\|x\|_{E(\mathcal M; \ell_\infty^\theta)}=\inf\Big\{\|a\|_{E^{(\frac{1}{1-\theta})}} \sup_{n\geq 1} \|y_n\|_\infty \|b\|_{E^{(\frac{1}{\theta})}}\Big\}$$
where the infimum is taken over all factorizations as above. We should mention that the case $\theta=\frac12$ reduces to the symmetric space $E(\mathcal M; \ell_\infty)$. In the case $\theta=0$ or $1$, the spaces $E(\mathcal M; \ell_\infty^0)$ and $E(\mathcal M; \ell_\infty^1)$ will be denoted by $E(\mathcal M; \ell_\infty^r)$ and $E(\mathcal M; \ell_\infty^c)$, respectively. It can be verified that
$$\|x\|_{E(\mathcal M; \ell_\infty^r)}=\inf\big\{\|A\|_E: A\geq 0,\,|x_n^*|^2\leq A^2,\,\forall\,n\geq 1\big\}.$$
Similarly,
$$\|x\|_{E(\mathcal M; \ell_\infty^c)}=\inf\big\{\|A\|_E: A\geq 0,\,|x_n|^2\leq A^2,\,\forall\, n\geq 1\big\}.$$

The following fact will be frequently used. We leave the proof as an exercise for the interested reader.

\begin{fact}\label{key-fact} Observe that if $\|x\|_{E(\mathcal M; \ell_\infty^\theta)}<1$, then there exist $a\in E^{(\frac{1}{1-\theta})}(\mathcal M),\,\, b\in E^{(\frac{1}{\theta})}(\mathcal M)$ and $y=(y_n)_{n\geq1}\subset L_\infty(\mathcal M)$
such that $x=ayb$ and
$$\max\Big\{\|a\|_{E^{(\frac{1}{1-\theta})}},\,\, \sup_{n\geq 1}\|y_n\|_\infty, \,\,\|b\|_{E^{(\frac{1}{\theta})}}\Big\}< 1.$$
\end{fact}

It is easy to check that $\|\cdot\|_{E(\mathcal M; \ell_\infty^\theta)}$ satisfies the positive definiteness and the homogeneity. The lemma below shows that  $\|\cdot\|_{E(\mathcal M;\ell_\infty^\theta)}$ is a quasi norm.

\begin{lem}[{\cite[Page 59]{RWZ2021}}]\label{quasi-norm} Let $0\leq\theta\leq 1$ and let $E$ be a symmetric Banach function space.
For every $x=(x_n)_{n\geq 1}$ and $y=(y_n)_{n\geq 1}$ in $E(\M;\ell_\infty^\theta)$, we have
\[\|x+y\|_{E(\mathcal{M};\ell_\infty^\theta)} \leq 2\big( \|x\|_{E(\mathcal{M};\ell_\infty^\theta)} + \|y\|_{E(\mathcal{M};\ell_\infty^\theta)}\big).
\]
\end{lem}

Though $\|\cdot\|_{E(\mathcal M; \ell_\infty^\theta)}$ is just a quasi norm in general, the next lemma shows that it is equivalent to a norm under certain assumption. In the case $\theta=0,1$ and $E=L_p$, the below result is just \cite[Lemma 3.5]{Mu2003}.

\begin{lem}\label{norm}
Let $0\leq\theta\leq 1$ and let $E$ be a symmetric  Banach function space with Fatou norm.  If $E$ is $\max\{2\theta, 2(1-\theta)\}$-convex,  then $\|\cdot\|_{E(\mathcal M; \ell_\infty^\theta)}$ is equivalent to a norm. In other words, $E(\mathcal M; \ell_\infty^\theta)$ can be renormed.
\end{lem}
\begin{proof}
Note that $E$ has  Fatou norm, and is $\max\{2\theta, 2(1-\theta)\}$-convex. According to Proposition \ref{DDS2014}, $\|\cdot\|_{E^{(\frac{1}{2(1-\theta)})}}$ is a equivalent to a norm which we denote by $|||\cdot|||_{E^{(\frac{1}{2(1-\theta)})}}$. Similarly, $||\cdot||_{E^{(\frac{1}{2\theta})}}$ is equivalent to a norm $|||\cdot|||_{E^{(\frac{1}{2\theta})}}$. Therefore, for each $x=(x_n)_{n\geq1}\in E(\mathcal M; \ell_\infty^\theta)$, $\|x\|_{E(\mathcal M; \ell_\infty^\theta)}$ is equivalent to
$$|||x|||_{E(\mathcal M; \ell_\infty^\theta)}=\inf\Big\{|||a|||_{E^{(\frac{1}{1-\theta})}} \sup_{n\geq 1} \|y_n\|_\infty |||b|||_{E^{(\frac{1}{\theta})}}\Big\},$$
where the infimum is taken over all factorizations $x=ayb$ with $a\in E^{(\frac{1}{1-\theta})}(\mathcal M),\,\, b\in E^{(\frac{1}{\theta})}(\mathcal M)$ and $y=(y_n)_{n\geq1}\subset L_\infty(\mathcal M)$.

In  the following, we show that $|||\cdot|||_{E(\mathcal M; \ell_\infty^\theta)}$ is a norm. We verify only the triangle inequality here.  Take $x^{(1)},\, x^{(2)}\in E(\mathcal M;\ell_\infty^\theta)$.  According to the homogeneity of symmetric spaces, for any $\varepsilon>0$, there exist factorizations
$x^{(1)}=a_1 y^{(1)} b_1$ and $x^{(2)}=a_2 y^{(2)} b_2$
such that $\sup_{n\geq 1}\|y_n^{(1)}\|_\infty\leq1$, $\sup_{n\geq 1}\|y_n^{(2)}\|_\infty\leq 1$, and
\begin{equation}\label{equation-1}
\begin{split}
\max\Big\{|||a_1|||_{E^{(\frac{1}{1-\theta})}},\,\, \,\,|||b_1|||_{E^{(\frac{1}{\theta})}}\Big\}\leq \Big(|||x^{(1)}|||_{E(\mathcal M;\ell_\infty^\theta)}+\varepsilon\Big)^{
\frac 1 2},\\
\max\Big\{|||a_2|||_{E^{(\frac{1}{1-\theta})}}, \,\,|||b_2|||_{E^{(\frac{1}{\theta})}}\Big\}\leq \Big(|||x^{(2)}|||_{E(\mathcal M;\ell_\infty^\theta)}+\varepsilon\Big)^{\frac 1 2}.
\end{split}
\end{equation}
Then we define
$$\alpha=\big(a_1a_1^*+a_2a_2^*+\varepsilon\textbf{1}\big)^{\frac 1 2},\quad \beta=\big(b_1^*b_1+b_2^*b_2+\varepsilon\textbf{1}\big)^{\frac 1 2}.$$
Since for $k\in\{1,2\}$ we have $a_ka_k^* \leq \alpha^2$ and $b_k^*b_k\leq \beta^2$, there exist contractions $u_1$, $u_2$, $w_1$ and $w_2$ such that
$a_1=\alpha u_1$, $a_2=\alpha u_2$, $b_1=w_1 \beta$ and $b_2=w_2\beta$. It is clear that
$$x^{(1)}+ x^{(2)}=\alpha\big(u_1y^{(1)}w_1+u_2y^{(2)}w_2\big)\beta.$$
Note that $|||\cdot|||_{E^{(\frac{1}{2(1-\theta)})}}$ is a norm. It follows that
\begin{align*}
|||\alpha|||_{E^{(\frac{1}{1-\theta})}}&=|||a_1a_1^*+a_2a_2^*+\varepsilon \textbf{1}|||_{E^{(\frac{1}{2(1-\theta)})}}^{\frac 1 2}\\
&\leq \Big(|||a_1a_1^*|||_{E^{(\frac{1}{2(1-\theta)})}}+||| a_2a_2^*|||_{E^{(\frac{1}{2(1-\theta)})}}+\varepsilon\Big)^{\frac 1 2}\\
&\leq \Big(|||x^{(1)}|||_{E(\mathcal M;\ell_\infty^\theta)}+|||x^{(2)}|||_{E(\mathcal M;\ell_\infty^\theta)}+2\varepsilon\Big)^{\frac 1 2}.
\end{align*}
Similarly, we have
$$|||\beta|||_{E^{(\frac{1}{\theta})}}\leq \Big(|||x^{(1)}|||_{E(\mathcal M;\ell_\infty^\theta)}+|||x^{(2)}|||_{E(\mathcal M;\ell_\infty^\theta)}+2\varepsilon\Big)^{\frac 1 2}.$$
Also, it is clear that
$$\big(u_1y^{(1)}w_1+u_2y^{(2)}w_2\big)\otimes e_{1,1}= \left[\begin{array}{cc}  u_1 & u_2  \\ 0& 0\end{array}\right]\left[\begin{array}{cc}  y^{(1)} & 0  \\ 0& y^{(2)}\end{array}\right] \left[\begin{array}{cc}  w_1 & 0 \\ w_2 & 0\end{array}\right].$$
Note that $\alpha$, $\beta$ are invertible and
$$u_1u_1^*+u_2u_2^*=\alpha^{-1}(a_1a_1^*+a_2a_2^*)\alpha^{-1}\leq \textbf{1},$$
$$w_1^*w_1+w_2^*w_2=\beta^{-1}(b_1^*b_1+b_2^*b_2)\beta^{-1}\leq \textbf{1}.$$
Moreover, since $\sup_{n\geq 1}\|y_n^{(1)}\|_\infty\leq1$, $\sup_{n\geq 1}\|y_n^{(2)}\|_\infty\leq 1$, we have
$$\Big\|u_1y^{(1)}w_1+u_2y^{(2)}w_2\Big\|_{L_\infty(\mathcal M\overline\otimes \ell_\infty)}\leq 1.$$
Therefore, we have found a factorization $$x^{(1)}+x^{(2)}=\alpha\big(u_1y^{(1)}w_1+u_2y^{(2)}w_2\big)\beta$$ with
\begin{align*}
&|||\alpha|||_{E^{(\frac{1}{1-\theta})}} \Big\|u_1y^{(1)}w_1+u_2y^{(2)}w_2\Big\|_{L_\infty(\mathcal M\overline\otimes \ell_\infty)} |||\beta|||_{E^{(\frac{1}{\theta})}} \\
&\qquad\leq |||x^{(1)}|||_{E(\mathcal M;\ell_\infty^\theta)}+|||x^{(2)}|||_{E(\mathcal M;\ell_\infty^\theta)}+2\varepsilon.
\end{align*}
Letting $\varepsilon\to 0$, the desired assertion follows.
\end{proof}

\begin{prop}\label{key-prop-banach}
Let $0\leq\theta\leq 1$ and let $E$ be a symmetric  Banach function space with Fatou norm.  If $E$ is $\max\{2\theta, 2(1-\theta)\}$-convex,  then $E(\M;\ell_\infty^\theta)$ can be renormed to be a Banach space.
\end{prop}
\begin{proof}
In fact, we will prove that $E(\M;\ell_\infty^\theta)$ is a Banach space with respect to the norm $|||\cdot|||_{E(\M;\ell_\infty^\theta)}$. Without causing any confusion, we replace $|||\cdot|||$ by $\|\cdot\|$ for simplicity below.

Let $x^{(1)},\,x^{(2)},\cdots,x^{(k)},\cdots$ be a Cauchy sequence in $E(\mathcal M;\ell_\infty^\theta)$. We may find a subsequence $x^{(n_1)},\,x^{(n_2)},\cdots\,x^{(n_k)},\cdots$ such that $$\|x^{(n_{1})}\|_{E(\mathcal M;\ell_\infty^\theta)}< 4^{-1}$$
and
$$\|x^{(n_{k})}-x^{(n_{k-1})}\|_{E(\mathcal M;\ell_\infty^\theta)}< 4^{-k},\quad k\geq 2.$$
It is well known that, by the triangle inequality proved in Lemma \ref{norm}, to obtain the completeness, it suffices to show that
$\sum_{k}x^{(n_{k+1})}-x^{(n_{k})}$ belongs to $E(\mathcal M;\ell_\infty^\theta)$. Set
$$\mathbf{x}^{(1)}=x^{(n_{1})}$$
and
$$\mathbf{x}^{(k)}=x^{(n_{k})}-x^{(n_{k-1})},\quad k\geq 2.$$
Then, by Fact \ref{key-fact} there exist $a_k\in E^{(\frac{1}{1-\theta})}(\M)$, $b_k\in E^{(\frac{1}{\theta})}(\M)$ and $y^{(k)}\in L_\infty(\mathcal M\overline{\otimes}\ell_\infty)$ such that $\mathbf{x}^{(k)}=4^{-k} a_k y^{(k)} b_k$ with
$$\max\Big\{\|a_k\|_{E^{(\frac{1}{1-\theta})}},\,\, \|y^{(k)}\|_{L_\infty(\mathcal M\overline{\otimes}\ell_\infty)},\,\,\|b_k\|_{E^{(\frac{1}{\theta})}}\Big\}<1. $$
We define for $\varepsilon>0$
$$\mathsf S_r(\mathsf a)=\Big(\sum_{k=1}^\infty 2^{-k}a_{k}a_k^*+\varepsilon\textbf{1} \Big)^{\frac 1 2},\quad \mathsf S_c(\mathsf b)=\Big(\sum_{k=1}^\infty2^{-k}b_k^*b_k+\varepsilon\textbf{1} \Big)^{\frac 1 2},$$
where the series converge in $E^{(\frac{1}{1-\theta})}(\mathcal M)$ and $E^{(\frac{1}{\theta})}(\mathcal M)$, respectively. In fact, it follows from Lemma \ref{norm} that
$$\|\mathsf S_r(\mathsf a)\|_{E^{(\frac{1}{1-\theta})}}\leq \Big(\varepsilon+\sum_{k=1}^\infty 2^{-k}\big\|a_k\big\|_{E^{(\frac{1}{1-\theta})}}^2\Big)^{\frac 1 2}\leq (1+\varepsilon)^{\frac 1 2}.$$
 A similar estimate applies to $\mathsf S_c(\mathsf b)$:
$$\|\mathsf S_c(\mathsf b)\|_{E^{(\frac{1}{\theta})}}\leq (1+\varepsilon)^{\frac 1 2}.$$
Clearly, there exist contractions $\alpha_k$, $\beta_k$ such that
$$2^{-k/2} a_k=\mathsf S_r(\mathsf a)\alpha_k,\quad 2^{-k/2} b_k=\beta_k\mathsf S_c(\mathsf b).$$
We may write
$$\sum_{k=1}^\infty \mathbf{x}^{(k)} = \mathsf S_r(\mathsf a) \Big(\sum_{k=1}^\infty2^{-k} \alpha_{k}y^{(k)}\beta_{k}\Big)\mathsf S_c(\mathsf b).$$
For the middle term of the above expression, we have
\begin{align*}
\Big\|\sum_{k=1}^\infty2^{-k} \alpha_{k}y^{(k)}\beta_{k}\Big\|_{L_\infty(\mathcal M\overline{\otimes} \ell_\infty)}&\leq \sum_{k=1}^\infty\Big\|2^{-k} \alpha_{k}y^{(k)}\beta_{k}\Big\|_{L_\infty(\mathcal M\overline{\otimes} \ell_\infty)}\leq 1.
\end{align*}
Therefore,
$$\Big\|\sum_{k=1}^\infty \mathbf{x}^{(k)}\Big\|_{E(\mathcal M;\ell_\infty^\theta)}\leq \|\mathsf S_r(\mathsf a)\|_{E^{(\frac{1}{1-\theta})}} \Big\|\sum_{k=1}^\infty2^{-k} \alpha_{k}y^{(k)}\beta_{k}\Big\|_{L_\infty(\mathcal M\overline{\otimes} \ell_\infty)} \|\mathsf S_c(\mathsf b)\|_{E^{(\frac{1}{\theta})}}\leq 1+\varepsilon.$$
The proof is complete.
\end{proof}


We now turn our attention to the space $E(\mathcal M;\ell_1^\theta)$. Let $0\leq \theta\leq 1$ and let $E$ be a symmetric Banach function space such that $E^\times$ is $\max\{2\theta, 2(1-\theta)\}$-convex. We define
\begin{equation}\label{p-theta}
E_{1-\theta}:=\Big(\big[(E^\times)^{(\frac{1}{2(1-\theta)})}\big]^{\times}\Big)^{(2)},\quad E_{\theta}:=\Big(\big[(E^\times)^{(\frac{1}{2\theta})}\big]^{\times}\Big)^{(2)}.
\end{equation}
Then the space $E(\mathcal M; \ell_1^\theta)$ is defined to be the set of all sequences $x=(x_n)_{n\geq 1}$  in $E(\mathcal M)$ which can be decomposed as
\begin{equation}\label{l1-theta-factorization}
x_n=\sum_{k\geq 1} v_{n,k} w_{n,k},\quad n\geq 1
\end{equation}
for two families $v_{n,k}\in E_{1-\theta}(\mathcal M)$ and $w_{n,k}\in E_{\theta}(\mathcal M)$ satisfying
$$\Big(\sum_{n,k\geq 1}v_{n,k}v_{n,k}^*\Big)^{\frac 1 2}\in E_{1-\theta}(\mathcal M)\quad\mbox{and}\quad \Big(\sum_{n,k\geq 1}w_{n,k}^*w_{n,k}\Big)^{\frac 1 2}\in E_{\theta}(\mathcal M)$$
where the series converge in norm. For $x\in E(\mathcal M; \ell_1^\theta)$,  define
$$\|x\|_{E(\mathcal M; \ell_1^\theta)}=\inf\left\{\Big\|\Big(\sum_{n,k\geq 1}v_{n,k}v_{n,k}^*\Big)^{\frac 1 2}\Big\|_{E_{1-\theta}}\cdot \Big\|\Big(\sum_{n,k\geq 1}w_{n,k}^*w_{n,k}\Big)^{\frac 1 2}\Big\|_{E_{\theta}}\right\}$$
where the infimum is taken over all factorizations as above.
Obviously, the case $\theta=\frac 1 2$ reduces to the symmetric space $E(\mathcal M; \ell_1)$. In the case $\theta=0$ or $1$, the spaces $E(\mathcal M; \ell_1^0)$ and $E(\mathcal M; \ell_1^1)$ will be denoted by $E(\mathcal M; \ell_1^r)$ and $E(\mathcal M; \ell_1^c)$, respectively.

\begin{fact}\label{key-fact-2} If $\|x\|_{E(\mathcal M; \ell_1^\theta)}<1$, then there exist two families $v_{n,k}\in E_{1-\theta}(\mathcal M)$ and $w_{n,k}\in E_{\theta}(\mathcal M)$ such that $x_n=\sum_{k\geq 1} v_{n,k} w_{n,k} $ for $n\geq 1$ and
$$\max\left\{\Big\|\Big(\sum_{n,k\geq 1}v_{n,k}v_{n,k}^*\Big)^{\frac 1 2}\Big\|_{E_{1-\theta}},\,\, \Big\|\Big(\sum_{n,k\geq 1}w_{n,k}^*w_{n,k}\Big)^{\frac 1 2}\Big\|_{E_{\theta}}\right\}< 1.$$
\end{fact}

\begin{lem}\label{norm-2}
Let $0\leq \theta\leq 1$ and  let $E$ be a symmetric Banach function space such that $E^\times$ is $\max\{2\theta, 2(1-\theta)\}$-convex. Then $\|\cdot\|_{E(\mathcal M;\ell_1^\theta)}$ is equivalent to a norm.
\end{lem}
\begin{proof} Note that $E^\times$ has  Fatou property, and is $\max\{2\theta, 2(1-\theta)\}$-convex. According to Proposition \ref{DDS2014}, $(E^\times)^{(\frac{1}{2(1-\theta)})}$ and $(E^\times)^{(\frac{1}{2(\theta)})}$ can be renormed to be Banach spaces.
 Similar to Lemma \ref{norm}, $\|x\|_{E(\mathcal M;\ell_1^\theta)}$ is equivalent to
$$|||x|||_{E(\mathcal M; \ell_1^\theta)}=\inf\left\{|||\Big(\sum_{n,k\geq 1}v_{n,k}v_{n,k}^*\Big)^{\frac 1 2}|||_{E_{1-\theta}}\cdot |||\Big(\sum_{n,k\geq 1}w_{n,k}^*w_{n,k}\Big)^{\frac 1 2}|||_{E_{\theta}}\right\},$$
where the infimum is taken over all factorizations as in \eqref{l1-theta-factorization}, $|||\cdot|||_{E_{1-\theta}}$, $|||\cdot|||_{E_{\theta}}$ denote the norms that are equivalent to  $\|\cdot\|_{E_{1-\theta}}$, $\|\cdot\|_{E_{\theta}}$ (these notions are referred to \eqref{p-theta}), respectively.

We now show that $|||x|||_{E(\mathcal M; \ell_1^\theta)}$ is norm. For simplicity, in what follows, we write $\|\cdot\|$ instead of $|||\cdot|||$. The positive definiteness and the homogeneity can be checked easily. We only provide the proof of the triangle inequality. Take $x=(x_n)_{n\geq 1}$ and $y=(y_n)_{n\geq 1}$ in $E(\mathcal M;\ell_1^\theta)$. By Fact \ref{key-fact-2} and the homogeneity of symmetric spaces, for any $\varepsilon>0$, there exist  two families $v_{n,k}\in E_{1-\theta}(\mathcal M)$ and $w_{n,k}\in E_{\theta}(\mathcal M)$ such that for every $n\geq 1$, $x_n=\sum_{k\geq 1} v_{n,k} w_{n,k} $, and
$$\max\left\{\Big\|\Big(\sum_{n,k\geq 1}v_{n,k}v_{n,k}^*\Big)^{\frac 1 2}\Big\|_{E_{1-\theta}},\,\, \Big\|\Big(\sum_{n,k\geq 1}w_{n,k}^*w_{n,k}\Big)^{\frac 1 2}\Big\|_{E_{\theta}}\right\}\leq \Big(\|x\|_{E(\mathcal M;\ell_1^\theta)}+\varepsilon\Big)^{\frac 1 2}.$$
Similarly, there exist  two families $a_{n,k}\in E_{1-\theta}(\mathcal M)$ and $b_{n,k}\in E_{\theta}(\mathcal M)$ such that for every $n\geq 1$, $y_n=\sum_{k\geq 1} a_{n,k} b_{n,k}$ and
$$\max\left\{\Big\|\Big(\sum_{n,k\geq 1}a_{n,k}a_{n,k}^*\Big)^{\frac 1 2}\Big\|_{E_{1-\theta}},\,\, \Big\|\Big(\sum_{n,k\geq 1}b_{n,k}^*b_{n,k}\Big)^{\frac 1 2}\Big\|_{E_{\theta}}\right\}\leq \Big(\|y\|_{E(\mathcal M;\ell_1^\theta)}+\varepsilon\Big)^{\frac 1 2}.$$
Obviously, for every $n\geq 1$, we have
$$x_n+y_n= \sum_{k\geq 1} v_{n,k} w_{n,k}+\sum_{k\geq 1} a_{n,k} b_{n,k}.$$
Note that $\|\cdot\|_{E_{1-\theta}^{(\frac 1 2)}(\mathcal M)}$ is a norm. Therefore,
\begin{align*}
&\Big\|\Big(\sum_{n,k\geq 1}v_{n,k}v_{n,k}^*+\sum_{n,k\geq 1}a_{n,k}a_{n,k}^*\Big)^{\frac 1 2}\Big\|_{E_{1-\theta}}\\
&=\Big\|\sum_{n,k\geq 1}v_{n,k}v_{n,k}^*+\sum_{n,k\geq 1}a_{n,k}a_{n,k}^*\Big\|_{E_{1-\theta}^{(\frac 1 2)}}^{\frac 1 2}\\
&\leq \Bigg(\Big\|\sum_{n,k\geq 1}v_{n,k}v_{n,k}^*\Big\|_{E_{1-\theta}^{(\frac 1 2)}}+\Big\|\sum_{n,k\geq 1}a_{n,k}a_{n,k}^*\Big\|_{E_{1-\theta}^{(\frac 1 2)}}\Bigg)^{\frac 1 2}\\
&\leq \Big(\|x\|_{E(\mathcal M;\ell_1^\theta)}+\|y\|_{E(\mathcal M;\ell_1^\theta)}+2\varepsilon\Big)^{\frac 1 2}.
\end{align*}
Similarly, one can show that
$$\Big\|\Big(\sum_{n,k\geq 1}w_{n,k}^*w_{n,k}+\sum_{n,k\geq 1}b_{n,k}^*b_{n,k}\Big)^{\frac 1 2}\Big\|_{E_{1-\theta}}\leq \Big(\|x\|_{E(\mathcal M;\ell_1^\theta)}+\|y\|_{E(\mathcal M;\ell_1^\theta)}+2\varepsilon\Big)^{\frac 1 2}.$$
Combining the last two estimates and letting $\varepsilon\to 0$, the desired assertion follows.
\end{proof}

The proof of the following result is similar to that of Proposition \ref{key-prop-banach}. We include  details for the convenience of the reader.

\begin{prop}\label{key-prop-banach-2}
Let $0\leq \theta\leq 1$ and  let $E$ be a symmetric Banach function space such that $E^\times$ is $\max\{2\theta, 2(1-\theta)\}$-convex. Then $E(\mathcal M;\ell_1^\theta)$ can be renormed to be a Banach space.
\end{prop}

\begin{proof}
In fact, we will prove that $E(\M;\ell_1^\theta)$ is a Banach space with respect to the norm $|||\cdot|||_{E(\M;\ell_1^\theta)}$ coming from Lemma \ref{norm-2}. Without causing any confusion, we replace $|||\cdot|||$ by $\|\cdot\|$ for simplicity below.

Let $x^{(1)},\,x^{(2)},\cdots,x^{(k)},\cdots$ be a Cauchy sequence in $E(\mathcal M;\ell_1^\theta)$. We may find a subsequence $x^{(n_1)},\,x^{(n_2)},\cdots\,x^{(n_k)},\cdots$ such that $$\|x^{(n_{1})}\|_{E(\mathcal M;\ell_1^\theta)}< 4^{-1}$$
and
$$\|x^{(n_{k})}-x^{(n_{k-1})}\|_{E(\mathcal M;\ell_1^\theta)} < 4^{-k},\quad k\geq 2.$$
It is well known that, by the triangle inequality proved in Lemma \ref{norm-2}, to obtain the completeness, it suffices to show that
$\sum_{k}x^{(n_{k+1})}-x^{(n_{k})}$ belongs to $E(\mathcal M;\ell_1^\theta)$. Set
$$\mathbf{x}^{(1)}=x^{(n_{1})}$$
and
$$\mathbf{x}^{(k)}=x^{(n_{k})}-x^{(n_{k-1})},\quad k\geq 2.$$
Then by Fact \ref{key-fact}, there exist  two families $v_{n,j}^{(k)}\in E_{1-\theta}(\mathcal M)$ and $w_{n,j}^{(k)}\in E_{\theta}(\mathcal M)$ such that  $\mathbf{x}^{(k)}_n=4^{-k} \sum_{j} v_{n,j}^{(k)} w_{n,j}^{(k)}$ with
\begin{equation}\label{key-equation-2}
\max\left\{\Big\|\Big(\sum_{n,j\geq 1}v_{n,j}^{(k)} {v_{n,j}^{(k)} }^*\Big)^{\frac 1 2}\Big\|_{E_{1-\theta}},\,\, \Big\|\Big(\sum_{n,k\geq 1}{w_{n,j}^{(k)}}^*w_{n,j}^{(k)}\Big)^{\frac 1 2}\Big\|_{E_{\theta}}\right\}< 1.
\end{equation}
Using a similar factorization trick as in Proposition \ref{key-prop-banach}, we define for $\varepsilon>0$
$$\mathbf{V}_{n,j}=\Big(\sum_{k=1}^\infty 2^{-k}v_{n,j}^{(k)} {v_{n,j}^{(k)} }^*+\varepsilon\textbf{1} \Big)^{\frac 1 2},\quad \mathbf W_{n,j}=\Big(\sum_{k=1}^\infty2^{-k}{w_{n,j}^{(k)}}^*w_{n,j}^{(k)}+\varepsilon\textbf{1} \Big)^{\frac 1 2}.$$
We claim that
$$(\mathbf{V}_{n,j})_{n,j}\in E_{1-\theta}\big(\mathcal M;\ell_2^r(\mathbb N^2)\big),\qquad (\mathbf{W}_{n,j})_{n,j}\in E_{\theta}\big(\mathcal M;\ell_2^c(\mathbb N^2)\big).$$
Indeed,  the triangle inequality holds true in $E_{1-\theta}^{(\frac 1 2)}(\mathcal M)$. Thus,
\begin{align*}
\Big\|\Big(\sum_{n,j=1}^\infty \mathbf{V}_{n,j}\mathbf{V}_{n,j}^*\Big)^{\frac 1 2}\Big\|_{E_{1-\theta}}&= \Big\| \sum_{k=1}^\infty 2^{-k} \sum_{n,j=1}^\infty v_{n,j}^{(k)} {v_{n,j}^{(k)} }^*+\varepsilon\textbf{1}\Big\|_{E_{1-\theta}^{(\frac 1 2)}}^{\frac 1 2}\\
& \leq \Big(\sum_{k=1}^\infty2^{-k}\Big\| \sum_{n,j=1}^\infty v_{n,j}^{(k)} {v_{n,j}^{(k)} }^*\Big\|_{E_{1-\theta}^{(\frac 1 2)}}+\varepsilon\Big)^{\frac 1 2}.
\end{align*}
According to \eqref{key-equation-2}, it follows that
$$
\Big\|\Big(\sum_{n,j=1}^\infty \mathbf{V}_{n,j}\mathbf{V}_{n,j}^*\Big)^{\frac 1 2}\Big\|_{E_{1-\theta}}\leq (1+\varepsilon)^{\frac 1 2}.
$$
A similar argument can be applied to get that
$$\Big\|\Big(\sum_{n,j=1}^\infty \mathbf{W}_{n,j}^*\mathbf{W}_{n,j}\Big)^{\frac 1 2}\Big\|_{E_{\theta}}\leq (1+\varepsilon)^{\frac 1 2}.$$
This shows our claim. Obviously, there exist contractions $\alpha_{n,j}^{(k)}$, $\beta_{n,j}^{(k)}$ such that
$$2^{-k/2} v_{n,j}^{(k)} =\mathbf{V}_{n,j} \alpha_{n,j}^{(k)},\qquad 2^{-k/2} w_{n,j}^{(k)}=\beta_{n,j}^{(k)} \mathbf{W}_{n,j}.$$
It is clear that, for every $n\geq 1$, the $n$-th term of $\sum_{k=1}^\infty \mathbf{x}^{(k)}$ equals to
$$\sum_{k=1}^\infty \mathbf{x}_n^{(k)} = \sum_{j=1}^\infty \mathbf{V}_{n,j} \Big(\sum_{k= 1}^{\infty}2^{-k}\alpha_{n,j}^{(k)}\ \beta_{n,j}^{(k)} \Big) \mathbf{W}_{n,j}.$$
For fixed $n$, $j$, the middle term of the above expression is a contraction:
$$\Big\|\sum_{k= 1}^{\infty}2^{-k}\alpha_{n,j}^{(k)}\beta_{n,j}^{(k)}\Big\|_\infty\leq \sum_{k= 1}^{\infty}2^{-k}\| \alpha_{n,j}^{(k)}\beta_{n,j}^{(k)}\|_\infty\leq 1.$$
Therefore,
$$\|\sum_{k=1}^\infty \mathbf{x}^{(k)}\|_{E(\mathcal M;\ell_1^\theta)}\leq \Big\|\Big(\sum_{n,j=1}^\infty \mathbf{V}_{n,j}\mathbf{V}_{n,j}^*\Big)^{\frac 1 2}\Big\|_{E_{1-\theta}}\Big\|\Big(\sum_{n,j=1}^\infty \mathbf{W}_{n,j}^*\mathbf{W}_{n,j}\Big)^{\frac 1 2}\Big\|_{E_{\theta}}\leq 1+\varepsilon.$$
The proof is complete.
\end{proof}

\begin{rem}\label{dense-remark}
Let $\mathfrak{F}$ be the set of elements $x=(x_n)_{n\geq 1}$ with
$$x_n=\sum_{k\geq 1}v_{n,k}w_{n,k},\quad n\geq 1$$
and $$\emph{Card}\big\{(n,k): v_{n,k}\neq 0\,\,\textit{or}\,\, w_{n,k}\neq 0\big\}<\infty.$$
Then  $\mathfrak{F}$ is dense in $E(\mathcal M;\ell_1^\theta)$. This can be seen from the fact that the set of finite sequences is dense in $E_{1-\theta}(\mathcal M;\ell_2^r)$ and also in $E_\theta(\mathcal M;\ell_2^c)$.
\end{rem}

Let $\h_E^{1_\theta}$ (resp. $\h_E^{\infty_\theta}$) be the subspace of $E(\mathcal M; \ell_1^\theta)$ (resp. $E(\mathcal M; \ell_\infty^\theta)$) consisting of martingale difference sequences. We have the following complemented result.

\begin{prop}\label{complemented-lem} Let $0\leq \theta \leq 1$. Assume that $E$ is a symmetric Banach function space with $E\in{\rm Int}(L_p, L_q)$ for $1<p\leq q<\infty$.
\begin{enumerate}[\rm (i)]
\item If $E^\times$ is $\max\{2\theta,2(1-\theta)\}$-convex and $q'>\max\{2\theta,2(1-\theta)\}$, then $\h_E^{1_\theta}$ is complemented in $E(\mathcal M; \ell_1^\theta)$.

\item If $E$ is $\max\{2\theta, 2(1-\theta)\}$-convex with Fatou norm, and $p>\max\{2\theta,2(1-\theta)\}$, then $\h_E^{\infty_\theta}$ is complemented in $E(\mathcal M; \ell_\infty^\theta)$.
\end{enumerate}
\end{prop}

\begin{proof} Let us first show (i). By Proposition \ref{key-prop-banach-2}, $E(\mathcal M;\ell_1^\theta)$ can be renormed to be a Banach space. It therefore suffices to prove that the Stein projection
$$\mathcal D((x_n)_{n\geq 1})=(d x_n)_{n\geq 1}$$
is bounded on $E(\mathcal M; \ell_1^\theta)$. Let $x_n=\sum_{k} v_{n,k} w_{n,k}$ be the decomposition of $x_n$ as in  \eqref{l1-theta-factorization}. Then for each $n$ we may write
\begin{align*}
\mathcal E_n(x_n)\otimes e_{1,1}&=\sum_{k}\mathcal E_n (v_{n,k}w_{n,k})\otimes e_{1,1}\\
&=\sum_{k,\,j}u_n(v_{n,k}^*)(j)^* u_n(w_{n,k})(j).
\end{align*}
Note that $E_{1-\theta}^{(\frac 12)}\in \textrm{Int}(L_r,L_s)$ for some $1< r \leq s<\infty $ when   $q'>\max\{2\theta,2(1-\theta)\}$. Applying the dual Doob
inequality for symmetric spaces (\cite[Corollary 4.13]{Dirksen2}), we have
\begin{align*}
\Big\|\Big(\sum_{n,\,k,\,j}  |u_n(v_{n,k})(j)^*| \Big)^{\frac 1 2}\Big\|_{E_{1-\theta}}&=\Big\|\Big(\sum_{n,\,k}\mathcal E_n(v_{n,k}v_{n,k}^*)\Big)^{\frac 1 2}\Big\|_{E_{1-\theta}}\\
&\lesssim_E \Big\|\Big(\sum_{n,\,k}v_{n,k}v_{n,k}^*\Big)^{\frac 1 2}\Big\|_{E_{1-\theta}}.
\end{align*}
Similarly,
\begin{align*}
\Big\|\Big(\sum_{n,\,k,\,j}  | u_n(w_{n,k})(j)|  \Big)^{\frac 1 2}\Big\|_{E_{\theta}}&=\Big\|\Big(\sum_{n,\,k}\mathcal E_n(w_{n,k}^*w_{n,k})\Big)^{\frac 1 2}\Big\|_{E_{\theta}}\\ &\lesssim_E \Big\|\Big(\sum_{n,\,k}w_{n,k}^*w_{n,k}\Big)^{\frac 1 2}\Big\|_{E_{\theta}}.
\end{align*}
Therefore, $(\mathcal E_n(x_n))_{n\geq 1}\in E(\mathcal M;\ell_1^\theta)$, which shows that the space $\h_E^{1_\theta}$ is complemented in $E(\mathcal M;\ell_1^\theta)$.

Now we turn to prove (ii). Let $x_n=ay_nb$ be the decomposition of $x_n$ as in
\eqref{linfty-theta-factorization}.  It suffices to show that $(\mathcal E_n (ay_nb))_{n\geq1}\in E(\mathcal M;\ell_\infty^\theta)$. Without loss of generality, we may assume that   $\sup_n\|y_n\|_\infty\leq 1$. From \eqref{u}, we immediately have
$$\mathcal E_n (ay_nb)\otimes e_{1,1}= u_n(y_n^*a^*)^*u_n(b).$$
We claim that $(u_n(y_n^*a^*)^*)_{n\geq 1}\in E^{(\frac{1}{1-\theta})}(\mathcal M; \ell_\infty^r)$ and $(u_n(b))_{n\geq 1}\in E^{(\frac{1}{\theta})}(\mathcal M; \ell_\infty^c)$. Indeed, we have
\begin{align*}
u_n(y_n^*a^*)^*u_n(y_na^*)=\mathcal E_n(a|y_n^*|^2 a^*)\otimes e_{1,1}\leq \mathcal E_n(a a^*)\otimes e_{1,1}.
\end{align*}
Note that $E^{(\frac{1}{2(1-\theta)})}\in \textrm{Int} (L_m,L_n)$  for some $1<m\leq n<\infty$ when $p>\max\{2\theta,2(1-\theta)\}$.  Combining the above inequality with the Doob inequality (\cite[Theorem 5.7]{Dirksen2}), we obtain that $(u_n(y_na^*)^*)_{n\geq 1}\in E^{(\frac{1}{1-\theta})}(\mathcal M; \ell_\infty^r)$.
Similar arguments can be applied to prove that $(u_n(b))_{n\geq 1}\in E^{(\frac{1}{\theta})}(\mathcal M; \ell_\infty^c)$. It is obvious that
$$E(\mathcal M;\ell_\infty^\theta)= E^{(\frac{1}{1-\theta})}(\mathcal M; \ell_\infty^r)\odot E^{(\frac{1}{\theta})}(\mathcal M; \ell_\infty^c).$$
Therefore, we have $(\mathcal E_n (ay_nb))_{n\geq 1}\in E(\mathcal M;\ell_\infty^\theta)$. The assertion is verified.
\end{proof}

\section{An interpolation result}\label{complex-interpolation}
Our main result of this section is a new complex interpolation result on asymmetric vector valued spaces $E(\mathcal M;\ell_\infty^\theta)$, which plays an essential role in the proofs of asymmetric Burkholder and Johnson-Schechtman inequalities.

\begin{thm}\label{complex-interpolation-result-1}
Let $\theta_0, \theta_1$ be such that $0\leq \theta_0 < \theta_1\leq 1$. Let $E$ be a fully symmetric Banach function.  Suppose that $E$ is $2\max\big\{1-\theta_0, \theta_0,1-\theta_1,\theta_1\big\}$-convex  with Fatou norm. Let $0\leq \theta, \eta \leq 1$ be such that
$$\eta=\frac{\theta-\theta_0}{\theta_1-\theta_0}.$$
Then
$$ E(\mathcal M; \ell_\infty^\theta)
=\big[E(\mathcal M; \ell_\infty^{\theta_0}), E(\mathcal M; \ell_\infty^{\theta_1})\big]_\eta.$$
\end{thm}

\begin{proof}
From Proposition \ref{key-prop-banach}, by the assumptions that $E$ has Fatou norm and $E$ is $\max\big\{2(1-\theta_0),2\theta_0,2(1-\theta_1),2\theta_1\big\}$-convex, we know that $E(\mathcal M; \ell_\infty^{\theta_0})$ and $E(\mathcal M; \ell_\infty^{\theta_1})$ can be renormed to be Banach spaces. Therefore, $$\big[E(\mathcal M; \ell_\infty^{\theta_0}), E(\mathcal M; \ell_\infty^{\theta_1})\big]_\eta$$ is also a Banach space. For the sake of clarity, we split the proof into several steps.
\smallskip

\noindent \textit{Step 1.} The lower estimate is easy. Indeed, let
$$T_1: E^{(\frac{1}{1-\theta_0})}(\mathcal M)\times L_\infty(\mathcal M\overline{\otimes} \ell_\infty)\times E^{(\frac{1}{\theta_0})}(\mathcal M)\to E(\mathcal M,\ell_\infty^{\theta_0}),$$
$$T_2: E^{(\frac{1}{1-\theta_1})}(\mathcal M)\times L_\infty(\mathcal M\overline{\otimes} \ell_\infty)\times E^{(\frac{1}{\theta_1})}(\mathcal M)\to E(\mathcal M,\ell_\infty^{\theta_1})$$
be the maps given by
$$T_1(a,(y_n)_{n\geq 1},b)= (ay_nb)_{n\geq 1}= T_2(a,(y_n)_{n\geq 1},b).$$ Clearly, both $T_1$, $T_2$ are contractive and multilinear. Therefore, from Lemma \ref{complex-E-infty}, we have the following continuous embedding:
$$ E(\mathcal M; \ell_\infty^\theta)
\subset \big[E(\mathcal M; \ell_\infty^{\theta_0}), E(\mathcal M; \ell_\infty^{\theta_1})\big]_\eta.$$
\smallskip

\noindent\textit{Step 2.} We claim that it suffices to show that
\begin{equation}\label{key-inequality}
\|x\|_{E(\mathcal M; \ell_\infty^\theta)}\leq \|x\|_{[E(\mathcal M; \ell_\infty^{\theta_0}), E(\mathcal M; \ell_\infty^{\theta_1})]_\eta},\quad \forall\,x=(x_n)_{n\geq 1}\in L_\infty(\mathcal M\overline\otimes \ell_\infty).
\end{equation}
In fact, assume that the inequality \eqref{key-inequality} holds true.  Then, combining  inequality \eqref{key-inequality} with
the  estimate established in Step 1, we deduce that
$$\|x\|_{E(\mathcal M; \ell_\infty^\theta)}= \|x\|_{[E(\mathcal M; \ell_\infty^{\theta_0}), E(\mathcal M; \ell_\infty^{\theta_1})]_\eta},\quad \forall\,x=(x_n)_{n\geq 1}\in L_\infty(\mathcal M\overline\otimes \ell_\infty).$$
It is obvious that $L_\infty(\mathcal M\overline\otimes \ell_\infty)$ continuously embeds into $E(\mathcal M; \ell_\infty^\theta)$ as a dense subspace. According to Proposition \ref{key-prop-banach}, we know that
the space $E(\mathcal M; \ell_\infty^\theta)$ is a Banach space.  Therefore, it follows that
$$\|x\|_{E(\mathcal M; \ell_\infty^\theta)}= \|x\|_{[E(\mathcal M; \ell_\infty^{\theta_0}), E(\mathcal M; \ell_\infty^{\theta_1})]_\eta},\quad\forall\, x=(x_n)_{n\geq 1}\in E(\mathcal M; \ell_\infty^\theta).$$
This implies that $E(\mathcal M; \ell_\infty^\theta)$ isometrically embeds into $[E(\mathcal M; \ell_\infty^r), E(\mathcal M; \ell_\infty^c)]_\eta$. Since $ L_\infty(\mathcal M\overline\otimes \ell_\infty)$ is dense in the space $E(\mathcal M; \ell_\infty^{\theta_0})\cap E(\mathcal M; \ell_\infty^{\theta_1})$ and therefore dense in the space $[E(\mathcal M; \ell_\infty^{\theta_0}), E(\mathcal M; \ell_\infty^{\theta_1})]_\eta$, we further deduce that the space $E(\mathcal M; \ell_\infty^\theta)$ is norm dense in $[E(\mathcal M; \ell_\infty^{\theta_0}), E(\mathcal M; \ell_\infty^{\theta_1})]_\eta$.
Hence, $E(\mathcal M; \ell_\infty^\theta)$ and $[E(\mathcal M; \ell_\infty^{\theta_0}), E(\mathcal M; \ell_\infty^{\theta_1})]_\eta$ coincide, which proves the claim.

\smallskip

\noindent\textit{Step 3.} We now verify \eqref{key-inequality}. Let $x=(x_n)_{n\geq 1}\in L_\infty(\mathcal M\overline\otimes \ell_\infty)$  such that $$\|x\|_{[E(\mathcal M; \ell_\infty^{\theta_0}), E(\mathcal M; \ell_\infty^{\theta_1})]_\eta}< 1.$$
Then, there exists a bounded analytic function (here, $\mathcal{S}$ is referred to Section \ref{Complex})
$$f:\mathcal{S}\longrightarrow  E(\mathcal M; \ell_\infty^{\theta_0})+E(\mathcal M; \ell_\infty^{\theta_1})$$
 such that $f(\theta)=x$ and
\begin{equation}\label{boundary-condition-1}
\max\Big\{\sup_{z\in \partial_0}\|f(z)\|_{E(\mathcal M; \ell_\infty^{\theta_0})},\,\,\sup_{z\in \partial_1}\|f(z)\|_{E(\mathcal M; \ell_\infty^{\theta_1})}\Big\}< 1.
\end{equation}
From the boundary condition \eqref{boundary-condition-1} and the fact that $L_\infty(\mathcal M\overline\otimes \ell_\infty)$ can be seen as a  dense subspace of $E(\mathcal M; \ell_\infty^\theta)$ for any $0\leq \theta\leq 1$, we deduce that
$f|_{\partial\mathcal{S}}$ can be written  as follows:
$$f(z)=a(z)y(z)b(z), \; z\in\partial\mathcal{S},$$
where $a: \partial_j  \to  \M$, $b: \partial_j  \to  \M$ and $y:\partial_j\to L_\infty(\mathcal M\overline\otimes \ell_\infty)$ with $j=0,1$ satisfying the following estimates
\begin{equation}\label{boundary-condition-2}
\begin{split}
&\sup_{z\in\partial_0}\max\Big\{\|a(z)\|_{E^{(\frac{1}{1-\theta_0})}},\,\|y(z)\|_{L_\infty(\mathcal M\overline\otimes \ell_\infty)},\,\|b(z)\|_{{E^{(\frac{1}{\theta_0})}}}\Big\}< 1,\\
&\sup_{z\in\partial_1}\max\Big\{\|a(z)\|_{E^{(\frac{1}{1-\theta_1})}},\,\|y(z)\|_{L_\infty(\mathcal M\overline\otimes \ell_\infty)},\,\|b(z)\|_{{E^{(\frac{1}{\theta_1})}}}\Big\}< 1.
\end{split}
\end{equation}

Fix $\varepsilon>0$.  Define $\varphi_\varepsilon: \partial \mathcal{S}\rightarrow \M$ as $\varphi_\varepsilon (z)= a(z)a(z)^*  +\varepsilon\textbf{1}$. By  Devinatz's factorization theorem (see \cite[Theorem 2.2]{Pisier1995}, \cite{Devinatz1961}) using a conformal mapping from $\mathcal{S}$ onto the unit disc $\mathbb{D}$, there exists an analytic function $\alpha: \mathcal{S} \rightarrow \M$ satisfying that $\alpha^{-1}$ exists and it is also bounded analytic on $\mathcal{S}$, and
such that $\varphi_\varepsilon =\alpha^*\alpha$ on $\partial S$, i.e., we have the following identity on the boundary of $\mathcal{S}$:
\begin{equation}\label{boundary-functions 1}
{\alpha}(z)^*{\alpha}(z)=a(z)a(z)^* +\varepsilon\textbf{1}.
\end{equation}
Similarly, we may find a bounded analytic function $\beta:\mathcal S\to \M$ with bounded analytic inverse $ \beta^{-1}$  satisfying the following identities on the boundary:
\begin{equation}\label{boundary-functions 2}
{\beta}(z)^*{\beta}(z)=b(z)b(z)^* +\varepsilon\textbf{1}.
\end{equation}


Now we may write
$$f(z)=\alpha(z)\widetilde{y}(z)\beta(z),\quad  z\in \mathcal{S}$$
with
$$\widetilde{y}(z)={\alpha}^{-1}(z)f(z)  \beta^{-1}(z).$$
Note that $\widetilde{y}: \mathcal{S}\longrightarrow  E(\mathcal M; \ell_\infty^{\theta_0})+E(\mathcal M; \ell_\infty^{\theta_1})$ is bounded analytic.
From \eqref{boundary-functions 1} and \eqref{boundary-functions 2}, we conclude the following estimates:
\begin{align*}
&\sup_{z\in\partial_0}\big\|\alpha(z)\big\|_{E^{(\frac{1}{1-\theta_0})}}^2  \leq \sup_{z\in\partial_0}\big\|a(z)\big\|_{E^{(\frac{1}{1-\theta_0})}}^2 +\varepsilon< 1+\varepsilon,\\
&\sup_{z\in\partial_1}\big\|\alpha(z)\big\|_{E^{(\frac{1}{1-\theta_1})}}^2  \leq \sup_{z\in\partial_1}\big\|a(z)\big\|_{E^{(\frac{1}{1-\theta_1})}}^2 +\varepsilon< 1+\varepsilon,\\
&\sup_{z\in\partial_0}\big\| \beta(z)\big\|_{E^{(\frac{1}{\theta_0})}}^2  \leq \sup_{z\in\partial_0}\big\|b(z)\big\|_{E^{(\frac{1}{\theta_0})}}^2 +\varepsilon< 1+\varepsilon,\\
&\sup_{z\in\partial_1}\big\| \beta(z)\big\|_{E^{(\frac{1}{\theta_1})}}^2 \leq \sup_{z\in\partial_1}\big\|b(z)\big\|_{E^{(\frac{1}{\theta_1})}}^2 +\varepsilon< 1+\varepsilon.
\end{align*}
According to Lemma \ref{complex-E-infty}  (note that Lemma \ref{fully-symmetric-lem} assures that the related symmetric Banach spaces are fully symmetric), the above estimates imply that
\begin{align*}
\|\alpha(\theta)\|_{E^{(\frac{1}{1-\theta})}}
< (1+\varepsilon)^{\frac 1 2},\qquad
\|\beta(\theta)\|_{E^{(\frac{1}{\theta})}}< (1+\varepsilon)^{\frac 1 2}.
\end{align*}
At the same time, it is clear that
\begin{align*}
\| {\alpha}^{-1}(z)a(z)\|_\M^2< 1,\quad \|b(z) \beta^{-1} (z)\|_\M^2 < 1,\; \forall z\in \partial\mathcal{S}.
\end{align*}
Therefore,
\begin{align*}
 &\sup_{z\in\partial_0}\|\widetilde{y}(z)\|_{L_\infty(\mathcal M\overline\otimes \ell_\infty)}\\&= \sup_{z\in\partial_0}\|\alpha^{-1}(z) a(z) y(z) b(z) \beta^{-1} (z)\|_{L_\infty(\mathcal M\overline\otimes \ell_\infty)}\\
 &\leq \sup_{z\in\partial_0} \|\alpha^{-1}(z)a(z)\|_\M\|y(z)\|_{L_\infty(\mathcal M\overline\otimes \ell_\infty)}\|b(z) \beta^{-1}(z)\|_\M\\
 &<  1
 \end{align*}
 and similarly,
 \begin{align*}
 \sup_{z\in\partial_1}\|\widetilde{y}(z)\|_{L_\infty(\mathcal M\overline\otimes \ell_\infty)}< 1.
\end{align*}
Then it is easy to see that
$$\|\widetilde{y}(\theta)\|_{L_\infty(\mathcal M\overline\otimes \ell_\infty)}< 1.$$
In conclusion, we have  the following factorization of $x$:
$$x=f(\theta)={\alpha}(\theta)\widetilde y(\theta)  \beta(\theta)$$
with
$$\|{\alpha}(\theta) \|_{E^{(\frac{1}{1-\theta})}}
< (1+\varepsilon)^{\frac 1 2}
,\quad \|\widetilde y(\theta)\|_{L_\infty(\mathcal M\overline\otimes \ell_\infty)}< 1,\quad \| {\beta} (\theta) \|_{E^{(\frac{1}{\theta})}}< (1+\varepsilon)^{\frac 1 2}.$$
Hence, we get
$$\|x\|_{E(\mathcal M; \ell_\infty^\theta)}< 1+\varepsilon.$$
Letting $\varepsilon\to 0$ in the above inequality, we obtain the desired result. The proof is complete.
\end{proof}

%
%

\begin{rem}
We should point out that, if $E=L_p$, then the above theorem goes back to a particular case of \cite[Theorem A]{JP2010}. Moreover, if $\theta=0,1$ and $E=L_p$ the above theorem goes back to \cite[Proposition 3.7]{Mu2003}.
\end{rem}




\section{Asymmetric Burkholder inequalities}\label{sec-asymmetric-burkholer}
This section is devoted to proving Theorem \ref{discrete-no-convex} and Theorem \ref{discrete-version-E}.

\subsection{Proof of Theorem \ref{discrete-no-convex}}

Our proof depends on the Cuculescu projections for self-adjoint martingales; see e.g. \cite[Proposition 1.4]{PR2006}.
Let $y\in L_1(\mathcal{M})$ be a self-adjoint operator. Let $R_{-1}^{\lambda}=1$ for $\lambda\in\mathbb{R}$ and define by induction
\begin{equation}\label{cuc p}
	R_n^{\lambda}=R_{n-1}^{\lambda}\chi_{(-\infty,\lambda)}(R_{n-1}^{\lambda}\mathcal{E}_n(y)R_{n-1}^{\lambda}),\quad\quad n\geq0.
\end{equation}
It is obvious that $(R_n^{\lambda})_{n\geq0}$ is decreasing.
Let
\begin{equation}\label{RQ}
	R^{\lambda}=\bigwedge_{n\geq0}R_n^{\lambda},
\end{equation}
and let
$Q_n^{\lambda}= R_{n-1}^{\lambda}-R_n^{\lambda},$ $n\geq1$. Then
\begin{equation}\label{sum-Q}
	\sum_{n\geq1}Q_n^{\lambda}=1-R^{\lambda}.
\end{equation}


The following result can be obtained by  the combination of \cite[Proposition 3.4]{JZZ} and   \cite[(3.8)]{JZZ}.
\begin{lem}
	For every $y=y^*\in L_2(\mathcal{ M})$ and  $\lambda>0,$ we have
	$$\tau(1-R^{2\lambda})\leq 12\sum_{n\geq1}\|Q_n^{\lambda}dy_nQ_n^{\lambda}\|_{2}^2+12\sum_{n\geq1}\|Q_n^{\lambda}(y-\mathcal{E}_n(y))\|_{2}^2.$$
\end{lem}

Before going further, we present two basic estimates.

\begin{lem}
	For every $y=y^*\in L_2(\mathcal{ M})$ and  $\lambda>0,$ we have
	\begin{equation*}
		\sum_{n\geq0}\|Q_n^{\lambda}(y-\mathcal{E}_n(y))\|_{2}^2\leq\tau ((1-R^{\lambda})s(y)^2),
	\end{equation*}
where $s(y):=s_c(y)$.
\end{lem}
\begin{proof}
	Note that
	$$\mathcal{E}_n\Big((y-\mathcal{E}_n(y))^2\Big)=\mathcal{E}_n\Big((\sum_{k>n}dy_k)^2\Big)=\mathcal{E}_n\Big(\sum_{k_1,k_2>n}dy_{k_1}dy_{k_2}\Big)=\sum_{k>n}\mathcal{E}_n(dy_k^2).$$
Since $Q_n^{\lambda}\in \mathcal{M}_n$, it follows that
	\begin{align*}
		\|Q_n^{\lambda}(y-\mathcal{E}_n(y))\|_{2}^2&=\tau\Big(Q_n^{\lambda}\cdot \mathcal{E}_n((y-\mathcal{E}_n(y))^2) \cdot Q_n^{\lambda}\Big)\\
		&=\sum_{k>n}\tau\Big(Q_n^{\lambda}\cdot \mathcal{E}_n(dy_k^2) \cdot Q_n^{\lambda}\Big)=\sum_{k>n}\tau\Big(Q_n^{\lambda}\cdot \mathcal{E}_{k-1}(dy_k^2) \cdot Q_n^{\lambda}\Big)\\
		&=\tau\Big(Q_n^{\lambda}\cdot \sum_{k>n}\mathcal{E}_{k-1}(dy_k^2) \cdot Q_n^{\lambda}\Big)\leq \tau(Q_n^{\lambda}\cdot s(y)^2\cdot Q_n^{\lambda}),
	\end{align*}
	which, combining with \eqref{sum-Q}, further implies
	the desired inequality.
\end{proof}

\begin{lem}\label{lem-fac}
	Let $\theta\in [0,1]$, and $y=y^*\in L_2(\mathcal{M})$. Assume that the martingale difference sequence $(dy_n)_{n\geq1}$ has a factorization $dy_n=av_nb$,  $n\geq1$, with $\sup_{n\geq1}\|v_n\|_{\infty}\leq 1$. Then
	$$\sum_{n\geq1}\|Q_n^{\lambda}dy_nQ_n^{\lambda}\|_{2}^2\leq \theta \tau((1-R^{\lambda})|a^*|^{2/\theta}) +(1-\theta)\tau((1-R^{\lambda})|b|^{2/(1-\theta)}).$$
\end{lem}
\begin{proof}
	Note that
	$$\frac{1}{2}=\frac{\theta}{2}+\frac{1-\theta}{2}.$$
	By the H\"{o}lder inequality, we have, for each $n\geq1$,
	\begin{align*}
		\|Q_n^{\lambda}dy_nQ_n^{\lambda}\|_{2}^2&=\|Q_n^{\lambda}av_nbQ_n^{\lambda}\|_{2}^2\leq \|Q_n^{\lambda}a\|_{{\frac{2}{\theta}}}^2\|bQ_n^{\lambda}\|_{{\frac{2}{1-\theta}}}^2\\
		&=\|Q_n^{\lambda}|a^*|^2Q_n^{\lambda}\|_{{\frac{1}{\theta}}} \|Q_n^{\lambda}|b|^2Q_n^{\lambda}\|_{{\frac{1}{1-\theta}}}.
	\end{align*}
	Recall that the Hansen inequality (see \cite[Page 249]{Ha1979}) states that for bounded operator $B$ and positive operator $A$,
	$$B^*AB\leq (B^*A^pB)^{\frac{1}{p}}, \quad \forall\, p\geq1.$$
	From this, we have
	$$Q_n^{\lambda}|a^*|^2Q_n^{\lambda}\leq (Q_n^{\lambda}|a^*|^{2/\theta}Q_n^{\lambda})^\theta,\quad Q_n^{\lambda}|b|^2Q_n^{\lambda}\leq (Q_n^{\lambda}|b|^{2/(1-\theta)}Q_n^{\lambda})^{1-\theta}.$$
	Therefore,
	\begin{align*}	\|Q_n^{\lambda}dy_nQ_n^{\lambda}\|_{2}^2&\leq \|Q_n^{\lambda}|a^*|^{2/\theta}Q_n^{\lambda}\|_{1}^{\theta}~ \|Q_n^{\lambda}|b|^{2/(1-\theta)}Q_n^{\lambda}\|_{1}^{1-\theta}\\
		&\leq \theta \|Q_n^{\lambda}|a^*|^{2/\theta}Q_n^{\lambda}\|_{1}+ (1-\theta) \|Q_n^{\lambda}|b|^{2/(1-\theta)}Q_n^{\lambda}\|_{1}\\
		&=\theta \tau(Q_n^{\lambda}|a^*|^{2/\theta})+ (1-\theta) \tau(Q_n^{\lambda}|b|^{2/(1-\theta)}),
	\end{align*}
	where the second inequality is due to the Young inequality. Then, the desired result follows from \eqref{sum-Q}.
\end{proof}

The following result can be deduced from the above last three lemmas.

\begin{cor}\label{cor-fac}
	Let $\theta\in [0,1]$, and $y=y^*\in L_2(\mathcal{M})$. Assume that the martingale difference sequence $(dy_n)_{n\geq1}$ has a factorization $dy_n=av_nb$,  $n\geq1$, with $\sup_{n\geq1}\|v_n\|_{\infty}\leq 1$. Then,
	for every $y=y^*\in L_2(\mathcal{ M})$ and  $\lambda>0,$ we have
	$$\tau(1-R^{2\lambda})\leq 12 \tau((1-R^{\lambda})A_{\theta}^2), $$
	where
	$$A_{\theta} =\big[  \theta  |a^*|^{2/\theta}+(1-\theta)   |b|^{2/(1-\theta)}+ s(y)^2\big]^{\frac{1}{2}}.$$
\end{cor}

\begin{lem}\label{Young-pre}
	Let $\theta\in (0,1)$ and  $x=(x_n)_{n\geq1}\in E(\mathcal{M},\ell_{\infty}^{\theta})$. For any $\varepsilon>0$, there exists a factorization
	$$x_n=av_n b, \quad n\geq 1$$
	such that $\sup_{n \geq 1}\|v_n\|_{\infty}\leq 1$,
	$$\|a\|_{E^{(\frac{1}{\theta})}} \|b\|_{E^{(\frac{1}{1-\theta})}}\leq \|x\|_{E(\mathcal{M},\ell_{\infty}^{\theta})}+\varepsilon$$
	and
	$$\|a\|_{E^{(\frac{1}{\theta})}}^{\frac{1}{\theta}} = \|b\|_{E^{(\frac{1}{1-\theta})}}^{\frac{1}{1-\theta}}.$$
\end{lem}
\begin{proof}
	Take $\varepsilon>0$. According to the definition of $\|\cdot\|_{E(\mathcal{M},\ell_{\infty}^{\theta})}$, there exists a factorization
	$x_n=cv_n d$,  $n\geq 1$,
	such that
	$$\|c\|_{E^{(\frac{1}{\theta})}} \|d\|_{E^{(\frac{1}{1-\theta})}}=:AB=Y\leq \|x\|_{E(\mathcal{M},\ell_{\infty}^{\theta})}+\varepsilon.$$
	Setting
	$$a:= \frac{Y^{\theta}}{A} c, \quad b:=\frac{Y^{1-\theta}}{B} d,$$
the desired assertion follows.
\end{proof}

Now we are ready to provide the proof of Theorem \ref{discrete-no-convex}.  We also need a result proved in \cite[Theorem 3.8]{JZZ}.  Let $\Phi: \mathbb R \to \mathbb R_+$ be an Orlicz function, that is, $\Phi$ is an even convex function such that $\Phi(0)=0$ and $\Phi(\infty)=\infty.$ Then $L_{\Phi}(\mathcal{M})$ can be defined according to \eqref{E-def}.
Given $1\leq p\leq q\leq \infty,$ an Orlicz function $\Phi$ is said to be $p$-convex if the function $t\mapsto\Phi(t^{1/p})$, $t>0$, is convex; and $\Phi$ is said to be $q$-concave if the function $t\mapsto\Phi(t^{1/q})$, $t>0$, is concave.
Suppose that $\Phi$ is a $p$-convex and $q$-concave Orlicz function with $2<p\leq q<\infty$. Let $y\in L_{\Phi}(\mathcal{M})$ be self-adjoint and $0\leq A \in L_{\Phi}(\mathcal{M})$  satisfy
\begin{equation*}
	\tau(1-R^{{2\lambda}}) \leq \lambda^{-2}\tau((1-R^{\lambda})A^2), \quad \lambda>0.
\end{equation*}
Then
\begin{equation}\label{phi-estimate}
	\tau(\Phi(|y|))\leq c_{p,q} \tau(\Phi(A)).
\end{equation}
\begin{proof}[Proof of Theorem \ref{discrete-no-convex}]
By the definitions of $E(\mathcal M; \ell_\infty^{r})$, $E(\mathcal M; \ell_\infty^{c})$ and the Burkholder-Gundy inequality (see e.g. \cite[Theorem 1.3]{JSZZ} or \cite{Dirk-Pag-Pot-Suk}), we have
\begin{equation}\label{rE}
\|(dx_k)_{k\geq 1}\|_{E(\mathcal M; \ell_\infty^{r})}\leq \Big\|\Big(\sum_{k\geq 1}|dx_k^*|^2\Big)^\frac{1}{2}\Big\|_{E}\lesssim_E \|x\|_E
\end{equation}
and
$$\|(dx_k)_{k\geq 1}\|_{E(\mathcal M; \ell_\infty^{c})}\leq \Big\|\Big(\sum_{k\geq 1}|dx_k|^2\Big)^\frac{1}{2}\Big\|_{E}\lesssim_E \|x\|_E.$$
Note that, for any $\theta\in [0,1]$, $E(\mathcal{M},\ell_{\infty}^c)\cap E(\mathcal{M},\ell_{\infty}^r) \subset E(\mathcal{M},\ell_{\infty}^{\theta})$. Then
$$\|(dx_k)_{k\geq 1}\|_{E(\mathcal M; \ell_\infty^{\theta})}\leq \max\{\|(dx_k)_{k\geq 1}\|_{E(\mathcal M; \ell_\infty^{c})},\|(dx_k)_{k\geq 1}\|_{E(\mathcal M; \ell_\infty^{r})}\}\lesssim_E \|x\|_E.$$
The symmetric Burkholder inequality \eqref{ncE-maximal} gives that
$$\max\{\|s_c(x)\|_E,\|s_r(x)\|_E\}\lesssim_E \|x\|_E.$$
Hence, we have
\begin{equation}\label{low-x}
	\max\{\|(dx_k)_{k\geq 1}\|_{E(\mathcal M; \ell_\infty^{\theta})}, \|(dx_k)_{k\geq 1}\|_{E(\mathcal M; \ell_\infty^{1-\theta})}, \|s_c(x)\|_E,\|s_r(x)\|_E\}\lesssim_E \|x\|_E.
\end{equation}

Now we show the inverse inequality, i.e.,
	$$\|x\|_E\lesssim_E \|(dx_k)_{k\geq1}\|_{E(\mathcal{M},\ell_{\infty}^{\theta})}+\|(dx_k)_{k\geq1}\|_{E(\mathcal{M},\ell_{\infty}^{1-\theta})}+\|s_c(x)\|_E+\|s_r(x)\|_E.$$
We only prove this inequality for the case $\theta\in (0,1)$ since  the case $\theta=0,1$ can be easily obtained with slight modification. 	To this end, we first write
	$$\|x\|_E\leq \|y\|_E+\|z\|_E,$$
	where
	$y=\mathrm{Re}(x)=(x+x^*)/2$ and $z=(x-x^*)/2i$.
 According to Lemma \ref{Young-pre}, for any $\varepsilon>0$, we can find a factorization
	$$dy_n=av_k b, \quad n\geq 1$$
	such that $\sup_{k \geq 1}\|v_k\|_{\infty}\leq 1$,
	$$\|a\|_{E^{(\frac{1}{\theta})}} \|b\|_{E^{(\frac{1}{1-\theta})}}\leq \|(dy_k)_{n\geq1}\|_{E(\mathcal{M},\ell_{\infty}^{\theta})}+\varepsilon,\quad\|a\|_{E^{(\frac{1}{\theta})}}^{\frac{1}{\theta}} = \|b\|_{E^{(\frac{1}{1-\theta})}}^{\frac{1}{1-\theta}}.$$
 Then,
	combining Corollary \ref{cor-fac} and \eqref{phi-estimate}, for any $p$-convex and $q$-concave Orlicz function $\Phi$,  we have,
\begin{equation*}
	\tau(\Phi(|y|))\leq c_{p,q} \tau(\Phi(A_{\theta})),
\end{equation*}
which, together with   Theorem 7.1 in \cite{KaltonSMS}, gives us that
\begin{align*}
	\|y\|_E^2&\lesssim_E \|A_{\theta}\|_{E}^2\lesssim_E   \|\theta  |a^*|^{2/\theta}\|_{E^{(\frac{1}{2})}}+\|(1-\theta)   |b|^{2/(1-\theta)}\|_{E^{(\frac{1}{2})}}+ \|s(y)^2\|_{E^{(\frac{1}{2})}}\\
	&=\theta \| a\|_{E^{(\frac{1}{\theta})}}^{\frac{2}{\theta}}+(1-\theta)  \| b\|_{E^{(\frac{1}{1-\theta})}}^{\frac{2}{1-\theta}}+ \|s(y)\|_{E }^2\\
	&= \| a\|_{E^{(\frac{1}{\theta})}}^{2}\times  \| b\|_{E^{(\frac{1}{1-\theta})}}^{2}+ \|s(y)\|_{E }^2\\
	&\leq \|(dy_k)_{n\geq1}\|_{E(\mathcal{M},\ell_{\infty}^{\theta})}+\|s(y)\|_{E }^2+\varepsilon,
\end{align*}
where the last equality we used the Young inequality. Note that $\|\cdot\|_{E(\mathcal{M},\ell_{\infty}^{\theta})} $ is at least a quasi-norm; see Lemma \ref{quasi-norm}. Then,
\begin{align*}
	\|y\|_E&\lesssim_E \|(dx_k)_{k\geq1}\|_{E(\mathcal{M},\ell_{\infty}^{\theta})}+\|(dx_k^*)_{k\geq1}\|_{E(\mathcal{M},\ell_{\infty}^{\theta})}+\|s_c(x)\|_E+\|s_r(x)\|_E\\
	&=\|(dx_k)_{k\geq1}\|_{E(\mathcal{M},\ell_{\infty}^{\theta})}+\|(dx_k)_{k\geq1}\|_{E(\mathcal{M},\ell_{\infty}^{1-\theta})}+\|s_c(x)\|_E+\|s_r(x)\|_E.
\end{align*}
Similarly, we have
$$\|z\|_E\lesssim_E \|(dx_k)_{k\geq1}\|_{E(\mathcal{M},\ell_{\infty}^{\theta})}+\|(dx_k)_{k\geq1}\|_{E(\mathcal{M},\ell_{\infty}^{1-\theta})}+\|s_c(x)\|_E+\|s_r(x)\|_E.$$
We finish the proof by combining the estimates of $\|y\|_E$ and $\|z\|_E.$
\end{proof}

\subsection{Proof of Theorem \ref{discrete-version-E}} In this subsection, we apply the interpolation theorem to prove Theorem \ref{discrete-version-E}.

\begin{proof}[Proof of Theorem \ref{discrete-version-E}]
According to \eqref{low-x}, we only need to show, for any $\theta \in [0,1]$,
$$\|x\|_{E}\lesssim_E \max\big\{\|x\|_{\h_E^c}, \|x\|_{\h_E^r}, \|x\|_{\h_E^{\infty_{\theta}}}\big\}.$$
 By \eqref{ncE-maximal}, we already have
\begin{equation}\label{Junge-Xu's result}
\|x\|_{E}\lesssim_E \max\big\{\|x\|_{\h_E^c}, \|x\|_{\h_E^r}, \|x\|_{\h_E^{\infty_{\frac 1 2}}}\big\}.
\end{equation}
For the case
$$\|x\|_{\h_E^{\infty_{\frac 1 2}}}< \max\big\{\|x\|_{\h_E^c}, \|x\|_{\h_E^r}\big\},$$
inequality \eqref{Junge-Xu's result} implies that
$$\|x\|_{E}\lesssim_E \max\big\{\|x\|_{\h_E^c}, \|x\|_{\h_E^r}\big\}.$$
Therefore,
$$\|x\|_{E}\lesssim_E \max\big\{\|x\|_{\h_E^c}, \|x\|_{\h_E^r}, \|x\|_{\h_E^{\infty_{\theta}}}\big\},\quad 0\leq \theta\leq 1.$$
We now turn to the case where
$$\max\big\{\|x\|_{\h_E^c}, \|x\|_{\h_E^r}\big\}\leq \|x\|_{\h_E^{\infty_{\frac 1 2}}}.$$
It follows from \eqref{Junge-Xu's result} that
\begin{equation}\label{E1/2}
\|x\|_E\lesssim_E \|x\|_{\h_E^{\infty_{\frac 1 2}}}=\|(dx_k)_{k\geq 1}\|_{E(\mathcal M; \ell_\infty)}.
\end{equation}
For ${\frac 1 2}<\theta\leq 1$, Theorem \ref{complex-interpolation-result-1} gives that
\begin{equation*}\label{case-1/2-1}
E(\mathcal M; \ell_\infty)=[E(\mathcal M; \ell_\infty^{r}), E(\mathcal M; \ell_\infty^\theta)]_{\frac{1}{2\theta}}.
\end{equation*}
Therefore,
$$\|(dx_k)_{k\geq 1}\|_{E(\mathcal M; \ell_\infty)}\leq \|(dx_k)_{k\geq 1}\|_{E(\mathcal M; \ell_\infty^{r})}^{1-\frac{1}{2\theta}}\|(dx_k)_{k\geq 1}\|_{E(\mathcal M; \ell_\infty^{\theta})}^{\frac{1}{2\theta}},$$
which, together with \eqref{E1/2} and \eqref{rE},  implies
$$\|x\|_E\lesssim_E \|(dx_k)_{k\geq 1}\|_{E(\mathcal M; \ell_\infty^{\theta})}=\|x\|_{\h_E^{\infty_\theta}}, \quad \frac 1 2<\theta \leq1.$$
As for $0\leq\theta\leq \frac 1 2$, by Theorem \ref{complex-interpolation-result-1}, we have
\begin{equation*}\label{case-0-1/2}
E(\mathcal M; \ell_\infty)=[E(\mathcal M; \ell_\infty^{\theta}), E(\mathcal M; \ell_\infty^c)]_\eta,\quad \eta=(1-2\theta)/[2(1-\theta)].
\end{equation*}
Using the same argument as above, we get
$$\|x\|_E\lesssim_E \|(dx_k)_{k\geq 1}\|_{E(\mathcal M; \ell_\infty^{\theta})}=\|x\|_{\h_E^{\infty_\theta}}, \quad 0\leq\theta<\frac 1 2.$$
Therefore, we conclude that
$$\|x\|_{E}\lesssim_E \max\big\{\|x\|_{\h_E^c}, \|x\|_{\h_E^r}, \|x\|_{\h_E^{\infty_{\theta}}}\big\},\quad 0\leq \theta\leq 1.$$
The proof is complete.
\end{proof}

%
%
%

\section{Duality: proof of Theorem \ref{duality-ns}}\label{main-duality}
In this section, we establish the following duality and then use the duality to prove Theorem \ref{duality-ns}. In particular, this duality is a common generalization of the symmetric duality  established by Junge in \cite{Ju} (see also \cite{Dirksen2}) as well as the column version proved by Junge and Perrin in \cite{JP2014}.

\begin{thm}\label{key-prop}
	Let $0\leq\theta\leq 1$ and let $E$ be a separable symmetric Banach function space with Fatou property. Suppose that $E^\times$ is $\max\big\{2(1-\theta),2\theta\big\}$-convex. Then
	$$(E(\mathcal M; \ell_1^\theta))^*=E^\times(\mathcal M; \ell_\infty^\theta)$$
	isometrically.
\end{thm}

In order to prove this theorem, we need some preparation. Recall that
$$E_{1-\theta}:=\Big(\big[(E^\times)^{(\frac{1}{2(1-\theta)})}\big]^{\times}\Big)^{(2)},\quad  E_{\theta}:=\Big(\big[(E^\times)^{(\frac{1}{2\theta})}\big]^{\times}\Big)^{(2)}.$$
The next lemma discusses the connection between $E_{1-\theta}$ (respectively, $E_{\theta}$) and $(E^{\times})^{(\frac{1}{1-\theta})}$ (respectively, $(E^{\times})^{(\frac{1}{\theta})}$).

\begin{lem}\label{factorization-E-theta}
	Let $0\leq\theta\leq 1$ and let $E$ be a  symmetric Banach function space with Fatou property. Suppose that $E^\times$ is $\max\big\{2(1-\theta),2\theta\big\}$-convex. We have
	\begin{enumerate}[\rm (i)]
		\item $E=E_{1-\theta}\odot E_\theta$.
		\item $L_2= E_{1-\theta}\odot (E^{\times})^{(\frac{1}{1-\theta})}=E_{\theta}\odot (E^{\times})^{(\frac{1}{\theta})}$.
	\end{enumerate}
\end{lem}

\begin{proof}
	(i): The case $\theta=\frac{1}{2}$ is trivial.  It suffices for us to verify the result for $\frac12<\theta\leq 1$ since the remaining case can be treated in a similar way.  Note that $E$ has the Fatou property. Then $	E=E^{\times\times}$, and hence, it follows from Lemma \ref{prod-1}(i) that
	\begin{align*}
		E=E^{\times\times}=\Big([(E^\times)^{(\frac{1}{2\theta})}]^{(2\theta)}\Big)^\times=\Big([(E^\times)^{(\frac{1}{2\theta})}]^\times\Big)^{(2\theta)}\odot L_{\frac{2\theta}{2\theta-1}}.
	\end{align*}
	Now using Lemma \ref{prod-1}(ii), we further get
	\begin{equation}\label{key-prod-2}E^{(\frac{1}{\theta})} = \Big([(E^\times)^{(\frac{1}{2\theta})}]^\times\Big)^{(2)} \odot L_{\frac{2}{2\theta-1}}=E_{\theta}\odot L_{\frac{2}{2\theta-1}}.
	\end{equation}
	At the same time, for $\theta>\frac12$, we have $\frac{1}{2(1-\theta)}>1$. Then, applying Lemma \ref{prod-1} again,
	\begin{equation}\label{key-prod-1}
		E_{1-\theta}=\Big(\big[(E^\times)^{(\frac{1}{2(1-\theta)})}\big]^{\times}\Big)^{(2)}= {(E^{\times\times})}^{(\frac{1}{(1-\theta)})}\odot L_{\frac{2}{2\theta-1}}=E^{(\frac{1}{(1-\theta)})}\odot L_{\frac{2}{2\theta-1}}.
	\end{equation}
	With the above argument, we conclude that
	$$E_{1-\theta}\odot E_\theta=E^{(\frac{1}{(1-\theta)})}\odot L_{\frac{2}{2\theta-1}}\odot E_\theta=E^{(\frac{1}{(1-\theta)})}\odot E^{(\frac{1}{\theta})}=E.$$
	
	We now turn to verify (ii). If $ \frac12 <\theta\leq 1$, then it follows from \eqref{key-prod-1}, Lemma \ref{prod-1}(ii) and the well-known Lozanovski\u{i} factorization theorem ($L_1=E\odot E^\times$) that
	\begin{align*}
		E_{1-\theta}\odot (E^{\times})^{(\frac{1}{1-\theta})} & =E^{(\frac{1}{(1-\theta)})}\odot L_{\frac{2}{2\theta-1}}\odot (E^{\times})^{(\frac{1}{1-\theta})}\\
		& = L_{\frac1{1-\theta}}\odot L_{\frac{2}{2\theta-1}}=L_2.
	\end{align*}
	Assume now  $0\leq \theta< \frac12$. Note that in this case $\frac12 <1-\theta\leq 1$. Similar to \eqref{key-prod-2}, we have
	$$E^{(\frac{1}{1-\theta})}= E_{1-\theta}\odot L_{\frac{2}{1-2\theta}}.$$
	By \eqref{key-prod-2} and the Lozanovski\u{i} factorization theorem, we have
	$$L_{\frac{1}{1-\theta}}= E^{(\frac{1}{1-\theta})}\odot (E^\times)^{(\frac{1}{1-\theta})}=  E_{1-\theta}\odot L_{\frac{2}{1-2\theta}}\odot(E^\times)^{(\frac{1}{1-\theta})}.$$
	At the same time, it is obvious that
	$$L_{\frac{1}{1-\theta}}=L_2\odot L_{\frac{2}{1-2\theta}}.$$
	Comparing the last two equations, we conclude that
	$$L_2= E_{1-\theta}\odot(E^\times)^{(\frac{1}{1-\theta})},$$
	which is the desired assertion. The proof of $L_2=E_{\theta}\odot (E^{\times})^{(\frac{1}{\theta})}$ is similar.
\end{proof}

\begin{lem}\label{useful-continuous-embedding}
	Let $0\leq\theta\leq 1$ and let $E$ be a  symmetric Banach function space with Fatou property. Suppose that $E^\times$ is $\max\big\{2(1-\theta),2\theta\big\}$-convex. We have
	$$\ell_1(E(\mathcal M)) \subset E(\mathcal M; \ell_1^\theta).$$
\end{lem}
\begin{proof}
	For a finite sequence $x=(x_n)_{n=1}^N$ with $x_n\in E(\M)$, we may write
	$$x= \sum_{n=1}^N x^{(n)}$$
	where
	$$x^{(n)}=(x^{(n)}_k)_{k\geq1}:=(\delta_{n,k} x_k)_{k\geq 1}$$
	with $\delta_{n,k}=1$ if $n=k$ and $\delta_{n,k}=0$ if $n\neq k$. From Lemma \ref{useful-factorization} and Lemma \ref{factorization-E-theta}(i), we see that for any $\varepsilon>0$ and  $k\geq 1$, there exist $a_k\in E_{1-\theta}(\mathcal M)^+$ and $b_k\in E_\theta(\mathcal M)^+$ such that $|x_k|=a_kb_k$ and
	\begin{align*}
		\|a_k\|_{E_{1-\theta}}\|b_k\|_{E_\theta}\leq (1+\varepsilon)\|x_k\|_E.
	\end{align*}
	Then we write
	$$v_{k}^{(n)}= \delta_{n,k} u_k a_k\quad \mbox{and}\quad w_{k}^{(n)}= \delta_{n,k} b_k$$
	where  $x_k=u_k |x_k|$ is the polar decomposition of $x_k$. It is obvious that
	$$x^{(n)}_k= v_{k}^{(n)} w_{k}^{(n)},\quad k\geq 1.$$
	Consider
	$$ v_{k,j}^{(n)}=\begin{cases}
		v_{k}^{(n)} & \mbox{if } j=k,\\
		0 & \mbox{if } j\neq k;
	\end{cases}\qquad w_{k,j}^{(n)}=\begin{cases}
		w_{k}^{(n)} & \mbox{if } j=k,\\
		0 & \mbox{if } j\neq k.
	\end{cases}$$
	Then we have
	$$x^{(n)}_k= \sum_{j\geq 1} v_{k,j}^{(n)} w_{k,j}^{(n)},\quad k\geq 1.$$
	Therefore,
	\begin{align*}
		\|x^{(n)}\|_{E(\mathcal M; \ell_1^\theta)}&\leq \Big\|\Big(\sum_{k,j}v_{k,j}^{(n)}{v_{k,j}^{(n)*}}\Big)^{\frac 1 2}\Big\|_{E_{1-\theta}}\cdot \Big\|\Big(\sum_{k,j }{w_{k,j}^{(n)*}} w_{k,j}^{(n)}\Big)^{\frac 1 2}\Big\|_{E_{\theta}}\\
		&= \Big\|\Big(\sum_{k}v_{k}^{(n)}{v_{k}^{(n)*}}\Big)^{\frac 1 2}\Big\|_{E_{1-\theta}}\cdot \Big\|\Big(\sum_{k}{w_{k}^{(n)*}} w_{k}^{(n)}\Big)^{\frac 1 2}\Big\|_{E_{\theta}}\\
		&=\| a_n u_n^*\|_{E_{1-\theta}} \cdot \|b_n\|_{E_\theta}\\
		&\leq (1+\varepsilon) \|x_n\|_E,
	\end{align*}
	which yields that
	$$\|x\|_{E(\mathcal M; \ell_1^\theta)}\leq \sum_{n=1}^N\|x^{(n)}\|_{E(\mathcal M; \ell_1^\theta)}\leq (1+\varepsilon)\sum_{n=1}^N \|x_n\|_E=(1+\varepsilon)\|x\|_{\ell_1(E(\mathcal M))}.$$
We finish the  proof   by letting $\varepsilon\to 0$.
\end{proof}

\begin{lem}\label{norm-dense} Let $1<q <\infty$. Suppose that $E$, $F$ are  symmetric  Banach function spaces which satisfy $L_q=E \odot F$. The following are true:
	\begin{enumerate}[{\rm (i)}]
		\item if $a$ is a positive operator in $E(\mathcal M)$  with ${\rm supp}\,a=\bf 1$, then $aF(\mathcal M)$ is norm dense in $L_q(\mathcal M)$;
		
		\item if $b$ is a positive operator in $E(\mathcal M)$ with ${\rm supp}\,b=e$, then $bF(\mathcal M)$ is norm dense in $eL_q(\mathcal M)$.
	\end{enumerate}
\end{lem}
\begin{proof}
	For item {\rm (i)}, consider the operator $\mathcal L_a : F(\mathcal M)\to L_q(\mathcal M)$ defined by $x\mapsto ax$. One can easily
	see that the adjoint $\mathcal L_a^* : L_{q'}(\mathcal M)\to F^\times(\mathcal M)$ is given by $y\mapsto ay$. Since ${\rm supp}\,a=\bf 1$, $\mathcal L_a^*$
	is one to one.  This implies that the range of
	$\mathcal L_a$ is weak*-dense in $L_q(\mathcal M)$.  Since $L_q(\mathcal M)$ is reflexive, $aF(\mathcal M) = \textrm{Im}(\mathcal L_a)$ is norm dense in $L_q(\mathcal M)$.
	
	Now we check item {\rm (ii)}.  Fix $c\in E(\mathcal M)^+$ with $ {\rm supp}\, c =\bf 1$. Consider the operator
	$$a'=b+({\bf 1}-e)c({\bf 1}-e).$$
	Then $a'$ is positive with ${\rm supp}\,a'=\bf 1$.
	If $z\in  eL_q(\mathcal M)$, then by {\rm (i)}, there is a sequence of operators $(x_n)_{n\geq 1}$ in $F(\mathcal M)$ such that
	$$a'x_n\rightarrow z.$$
	Note that $z$ is left-supported  in $e$. Therefore,
	$$ea'x_n=bx_n\rightarrow z.$$
	This shows that $bF(\mathcal M)$ is norm dense in $eL_q(\mathcal M)$.
\end{proof}

\begin{lem}\label{key-lem} Let  $E$ be a symmetric Banach function space with Fatou property. Suppose that $E^\times$ is $\max\big\{2(1-\theta),2\theta\big\}$-convex.
	Suppose that $(z_n)_{n\geq 1}$ is a sequence of operators in $E^\times(\mathcal M)$. Let $\alpha\in (E^\times)^{(\frac{1}{1-\theta})}(\mathcal M)$, $\beta\in (E^\times)^{(\frac{1}{\theta})}(\mathcal M)$ be positive. If for any $v\in E_{1-\theta}(\mathcal M)$, $w\in E_\theta(\mathcal M)$,
	\begin{equation}\label{assumption}
		|\tau(z_n^* vw)|\leq \|\alpha v\|_2\|w\beta\|_2,\quad n\geq 1,
	\end{equation}
	then $(z_n)_{n\geq 1}\in E^{\times}(\mathcal M; \ell_\infty^\theta)$. Moreover,
	$$\|(z_n)_{n\geq 1}\|_{E^\times(\mathcal M; \ell_\infty^\theta)}\leq \|\alpha\|_{(E^\times)^{(\frac{1}{1-\theta})}} \|\beta\|_{(E^\times)^{(\frac{1}{\theta})}}.$$
\end{lem}

\begin{proof}
	Denote by $q_\alpha$, $q_\beta$ the support projections of $\alpha$, $\beta$, respectively. We claim that
	$$z_n^*=q_\beta z_n^* q_\alpha,\quad n\geq 1.$$
	In fact, for any $v\in E_{1-\theta}(\mathcal M)$ and $w \in E_{\theta}(\mathcal M)$, the assumption \eqref{assumption} implies that
	$$|\tau(z_n^*(1-q_\alpha)vw )|\leq \tau\big(\alpha^2(1-q_\alpha)vv^*\big)^{\frac 1 2}\tau\big(w^*w\beta^2\big)^{\frac 1 2}=0$$
	and
	$$|\tau((1-q_\beta)z_n^*vw)|\leq \tau\big(\alpha^2 vv^*\big)^{\frac 1 2}\tau\big(w^*w(1-q_\beta)\beta^2\big)^{\frac 1 2}=0.$$
	By Lemma \ref{factorization-E-theta}(i), we get $z_n^*(1-q_\alpha)=0$ and $(1-q_\beta)z_n^*=0$, which implies the claim. Therefore, we may write
	$$z_n^*=q_\beta z_n^* q_\alpha=q_\beta \beta q_\beta \beta^{-1}z_n^* \alpha^{-1}q_\alpha \alpha q_\alpha=(\alpha y_n \beta)^*,\quad n\geq 1.$$
	where $y_n= q_\alpha \alpha^{-1} z_n \beta^{-1}q_\beta$ for $n\geq1$.
	
	Now it remains to check:
	\begin{equation}\label{estimate-infty}
		\|y_n\|_\infty=\| q_\alpha \alpha^{-1} z_n \beta^{-1}q_\beta \|_\infty\leq 1,\quad n\geq 1.
	\end{equation}
	Fix $n\geq 1$ and $x\in L_1(\mathcal M)$. We may write $x=x_1 x_2$ with $x_1\in L_2(\mathcal M)$, $x_2\in L_2(\mathcal M)$ and $\|x\|_1=\|x_1\|_2\|x_2\|_2$. By Lemma \ref{norm-dense}(ii) and Lemma \ref{factorization-E-theta}(ii), we may assume that $x_1=\alpha v$ and $x_2=w\beta$ for some $v\in E_{1-\theta}(\mathcal M)$ and some $w\in E_\theta(\mathcal M)$. Therefore,
	\begin{align*}
		|\tau(y_n^* x)|&=|\tau\big( q_\beta \beta^{-1} z_n^* \alpha^{-1} q_\alpha x\big)|= |\tau\big(q_\beta \beta^{-1} z_n^* \alpha^{-1}q_\alpha x_1 x_2\big)|\\
		& =|\tau (z_n^* v w)|\leq \|\alpha v\|_2\|w\beta\|_2=\|x_1\|_2\|x_2\|_2=\|x\|_1,
	\end{align*}
	which implies \eqref{estimate-infty} by duality. The proof is complete.
\end{proof}

We now  prove the  duality.

\begin{proof}[Proof of Theorem \ref{key-prop}] We first show that
	$$E^\times(\mathcal M; \ell_\infty^\theta)\subset(E(\mathcal M; \ell_1^\theta))^*.$$
	Let $z=(z_n)_{n\geq 1}$ be in $E^\times(\mathcal M; \ell_\infty^\theta)$. For any $\varepsilon>0$, there exist $a\in (E^\times)^{(\frac{1}{1-\theta})}(\mathcal M),\,\, b\in (E^\times)^{(\frac{1}{\theta})}(\mathcal M)$ and $y=(y_n)_{n\geq1}\subset L_\infty(\mathcal M)$
	such that
	$
	z_n=ay_n b$ for all $n\geq 1
	$, and
	$$\|a\|_{(E^\times)^{(\frac{1}{1-\theta})}} \sup_{n\geq 1} \|y_n\|_\infty \|b\|_{(E^\times)^{(\frac{1}{\theta})}}\leq (1+\varepsilon)\|z\|_{E^\times(\mathcal M; \ell_\infty^\theta)}.$$
	Define the mapping $\phi_z: E(\mathcal M; \ell_1^\theta)\to \mathbb C$ as follows:
	$$\phi_z(x)=\Big|\sum_{n\geq 1}\tau( x_n z_n^*)\Big|,\quad \forall x=(x_n)_{n\geq 1}\in E(\mathcal M; \ell_1^\theta).$$
	Suppose that $x=(x_n)_{n\geq 1}\in E(\mathcal M; \ell_1^\theta)$. By the definition of $E(\mathcal M; \ell_1^\theta)$, there are families $v_{n,k}\in E_{1-\theta}(\mathcal M)$ and $w_{n,k}\in E_\theta(\mathcal M)$ such that
	$$x_n=\sum_{k\geq 1} v_{n,k} w_{n,k},\quad \forall n\geq 1,$$ and
	$$\Big\|\Big(\sum_{n,k\geq 1}v_{n,k}v_{n,k}^*\Big)^{\frac 1 2}\Big\|_{E_{1-\theta}}\cdot \Big\|\Big(\sum_{n,k\geq 1}w_{n,k}^*w_{n,k}\Big)^{\frac 1 2}\Big\|_{E_{\theta}}\leq (1+\varepsilon)\|x\|_{E(\mathcal M; \ell_1^\theta)}.$$
	Therefore, we deduce from H\"{o}lder's inequality that
	\begin{align*}
		|\phi_z(x)|&=\Big|\sum_{n\geq 1}\tau(x_n z_n^*)\Big|=\Big|\sum_{n\geq 1}\tau\Big( \sum_{k\geq 1}v_{n,k} w_{n,k} z_n^*\Big)\Big|\\
		& = \Big|\sum_{n, k\geq 1}\tau\Big(v_{n,k} w_{n,k}b^*y_n^*a^*\Big)\Big|=\Big|\sum_{n, k\geq 1}\tau\Big( a^*v_{n,k} w_{n,k}b^* y_n^*\Big)\Big|\\
		&\leq \sup_{n\geq 1}\|y_n^*\|_\infty \sum_{n,k\geq 1}\|a^*v_{n,k} w_{n,k}b^*\|_{1}\\
		&\leq \sup_{n\geq 1}\|y_n\|_\infty \Big(\sum_{n,k\geq 1}\|a^*v_{n,k}\|_2^2\Big)^{\frac 1 2}\cdot \Big(\sum_{n,k\geq 1}\|w_{n,k}b^*\|_2^2\Big)^{\frac 1 2}\\
		&=\sup_{n\geq 1}\|y_n\|_\infty \Big\|a^*\sum_{n,k\geq 1}v_{n,k}v_{n,k}^* a\Big\|_1^{\frac 1 2}\cdot \Big\|b\sum_{n,k\geq 1}w_{n,k}^*w_{n,k} b^*\Big\|_1^{\frac 1 2}
	\end{align*}
	Using Lemma \ref{factorization-E-theta}(ii), we get
	\begin{align*}
		|\phi_z(x)|&\leq \sup_{n\geq 1}\|y_n\|_\infty \|a\|_{(E^\times)^{(\frac1{1-\theta})}} \Big\|\Big(\sum_{n,k\geq 1}v_{n,k}v_{n,k}^*\Big)^{\frac 1 2}\Big\|_{E_{1-\theta}}\cdot   \|b\|_{(E^\times)^{(\frac1{\theta})}} \Big\|\Big(\sum_{n,k\geq 1}w_{n,k}^*w_{n,k} \Big)^{\frac 1 2}\Big\|_{E_{\theta}}\\
		&\leq (1+\varepsilon)^2\|x\|_{E(\mathcal M; \ell_1^\theta)} \|z\|_{E^{\times}(\mathcal M; \ell_\infty^\theta)}, \quad \forall \varepsilon>0,
	\end{align*}
	which implies that
	$$\|\phi_z\|\leq \|z\|_{E^\times(\mathcal M; \ell_\infty^\theta)}.$$
	Hence, $E^\times(\mathcal M; \ell_\infty^\theta)\subset(E(\mathcal M; \ell_1^\theta))^*.$
	
	Now we show that all the continuous functionals on $E(\mathcal M; \ell_1^\theta)$ are in $E^\times(\mathcal M; \ell_\infty^\theta)$. This direction is more involved. Let $\phi: E(\mathcal M; \ell_1^\theta)\to \mathbb C$ be a norm one functional. By Lemma \ref{useful-continuous-embedding},
	$$\ell_1(E(\mathcal M)) \subset E(\mathcal M; \ell_1^\theta).$$
	Since $E$ is separable, we may assume that there exists $z=(z_n)_{n\geq 1}\in \ell_\infty(E^\times(\mathcal M))$ such that
	$$\phi(x)=\sum_{n\geq 1}\tau(x_n z_n^*),\quad \forall x=(x_n)_{n\geq 1}\in \ell_1(E(\mathcal M)).$$
	Define
	$$B_{1-\theta}=\big\{c\in (E^\times)^{(\frac{1}{2(1-\theta)})}(\mathcal M)^+: \|c\|_{(E^\times)^{(\frac{1}{2(1-\theta)})}}\leq 1\big\}$$
	and
	$$B_{\theta}=\big\{d\in  (E^\times)^{(\frac{1}{2\theta})}(\mathcal M)^+: \|d\|_{(E^\times)^{(\frac{1}{2\theta})}}\leq 1\big\}.$$
	According to Lemma \ref{factorization-E-theta}(ii) and Lemma \ref{prod-1}(ii), we have
	$$L_1=(E_{1-\theta})^{(\frac 1 2)}\odot(E^\times)^{(\frac{1}{2(1-\theta)})}=(E_{\theta})^{(\frac 1 2)}\odot(E^\times)^{(\frac{1}{2\theta})}$$
	Therefore $B_{1-\theta}$ is compact with respect to the $\sigma\big((E^\times)^{(\frac{1}{2(1-\theta)})}(\mathcal M),(E_{1-\theta})^{(\frac12)}(\mathcal M)\big)$-topology. Similarly, $B_\theta$ is compact when equipped with the $\sigma\big((E^\times)^{(\frac{1}{2\theta})}(\mathcal M),(E_\theta)^{(\frac12)}(\mathcal M)\big)$-topology. According to the definition of $E(\mathcal M; \ell_1^\theta)$, we have
	\begin{equation*}
		\begin{aligned}
			&\Big|\sum_{n, k\geq 1}\tau(z_n^* v_{n,k}w_{n,k})\Big| =
			\Big| \phi\Big[\Big(\sum_{j\geq 1} v_{n,j}w_{n,j}\Big)_{n\geq 1}\Big]\Big|\\
			&\quad \leq \Big\|\sum_{n,k\geq 1}v_{n,k}v_{n,k}^*\Big\|_{E_{1-\theta}^{(\frac{1}{2})}}^{\frac 1 2}\cdot \Big\|\sum_{n,k\geq 1}w_{n,k}^*w_{n,k}\Big\|_{ E_{\theta}^{(\frac{1}{2})}}^{\frac 1 2}\\
			&\quad \leq \frac{1}{2}\sup_{c,d}\Big\{\sum_{n,j\geq 1}
			\tau(v_{n,j}v_{n,j}^*c)+\sum_{n,j\geq 1} \tau(w_{n,j}^*w_{n,j}d):c\in B_{1-\theta}, d\in B_\theta\Big\}.
		\end{aligned}
	\end{equation*}
	The right hand side remains unchanged under multiplication with signs $\varepsilon_{n,j}$. Therefore,
	\begin{equation}\label{separation}
		\begin{aligned}
			& \sum_{n, k\geq 1}|\tau(z_n^* v_{n,k}w_{n,k})|\\
			&  \leq \frac{1}{2}\sup_{c,d}\Big\{\sum_{n,j\geq 1} \tau(v_{n,j}v_{n,j}^*c)+\sum_{n,j\geq 1} \tau(w_{n,j}^*w_{n,j}d): c\in B_{1-\theta},\,\, d\in B_\theta\Big\}.
		\end{aligned}
	\end{equation}
	
	For any finite sequences $v=(v_n)_{n\geq 1}$ in $ E_{1-\theta}(\mathcal M)$ and $w=(w_n)_{n\geq 1}$ in $E_\theta(\mathcal M)$, we define the function
	$$f_{v,w}(c,d):=\sum_{n\geq1}\tau(v_nv_n^*c)+\tau(w_n^*w_n d)-2|\tau(z_n^*v_nw_n)|,\quad c\in B_{1-\theta},\,\, d\in B_\theta.$$
	It is clear that $f_{v,w}$ is a real valued continuous function on $B_{1-\theta}\times B_\theta$. Moreover, from \eqref{separation}, it follows that
	\begin{equation*}
		\sup_{c,d} \Big\{f_{v,w}(c,d): c\in B_{1-\theta},\,\, d\in B_\theta\Big\}\geq 0.
	\end{equation*}
	Let $C$ be the set of all functions $f_{v,w}$ as above. Then $C$ is a cone. Indeed, if $\lambda\geq 0$ and $f_{v,w}\in C$, then $\lambda f_{v,w}=f_{\sqrt{\lambda} v, \sqrt{\lambda} w}\in C$. Moreover, if $f_{v,w}, f_{\tilde v, \tilde w}\in C$, then $f_{v,w}+f_{\tilde v, \tilde w}$ can be realized as $f_{v+\tilde v, w+\tilde w}$, where $v+\tilde v$ denotes the sequence starting with $v$ and followed by $\tilde v$.
	It is obvious that $C$ is disjoint from the cone
	$$C_-=\Big\{g\in C(B_{1-\theta}\times B_\theta): \sup g<0\Big\},$$
	where $C(B_{1-\theta}\times B_\theta)$ denotes all real valued continuous functions on $B_{1-\theta}\times B_\theta$. By  the Hahn-Banach separation theorem, there exists a measure $\mu$ on $B_{1-\theta}\times B_\theta$ and
	a scalar $t$ such that for $f\in C$ and $g\in C_-$,
	\begin{equation}\label{Hahn-Banach-Separation}
		\int_{B_{1-\theta}\times B_\theta}g d\mu<t\leq \int_{B_{1-\theta}\times B_\theta}fd\mu.
	\end{equation}
	Since $C$ and $C_-$ are cones, it follows that $t=0$ and $\mu$ is positive. By normalization, we may assume that $\mu$ is a probability measure. Now define positive operators $a$, $b$ by
	$$a=\int_{B_{1-\theta}\times B_\theta} c \,d\mu,\qquad b=\int_{B_{1-\theta}\times B_\theta} d \, d\mu.$$
	By convexity of $B_{1-\theta}$ and $B_\theta$, we deduce that $a\in B_{1-\theta}$ and $b\in B_\theta$. Let $(v_n)_{n\geq 1}$ in $ E_{1-\theta}(\mathcal M)$ and $(w_n)_{n\geq 1}$ in $E_{\theta}(\mathcal M)$ be finite sequences. From \eqref{Hahn-Banach-Separation}, we have
	\begin{align*}
		0&\leq \int_{B_{1-\theta}\times B_\theta}f_{v,w} d\mu
		\\&=\int_{B_{1-\theta}\times B_\theta}\sum_{n\geq1}\big[\tau(v_nv_n^*c)+\tau(w_n^*w_n d)-2|\tau(z_n^*v_nw_n)|\big]d\mu\\
		&=\int_{B_{1-\theta}\times B_\theta}\sum_{n\geq1}\big[\tau(v_nv_n^*c)+\tau(w_n^*w_n d)\big]d\mu-2\sum_{n\geq1}|\tau(z_n^*v_nw_n)|\\
		& = \sum_{n\geq1}\int_{B_{1-\theta}\times B_\theta}\big[\tau(v_nv_n^*c)+\tau(w_n^*w_n d)\big]d\mu-2\sum_{n\geq1}|\tau(z_n^*v_nw_n)|\\
		&=\sum_{n\geq1}\big[\tau(v_nv_n^*a)+\tau(w_n^*w_n b)-2|\tau(z_n^*v_nw_n)|\big].
	\end{align*}
	Therefore,
	\begin{equation*}
		\begin{aligned}
			\sum_{n\geq1}2|\tau(z_n^*v_nw_n)|&\leq \sum_{n\geq1}\big[\tau(v_nv_n^*a)+\tau(w_n^*w_n b)\big].
		\end{aligned}
	\end{equation*}
	Moreover, note that for any $r>0$, $$\sum_{n\geq1}2|\tau(z_n^*v_nw_n)|=\sum_{n\geq1}2|\tau(z_n^*(r^{\frac 1 2}v_n) (r^{-\frac 1 2}w_n))|.$$
	Repeating the arguments above, we conclude that
	\begin{equation*}
		\begin{aligned}
			\sum_{n\geq1}2|\tau(z_n^*v_nw_n)|&\leq r\sum_{n\geq1}\tau(v_nv_n^*a)+r^{-1}\sum_{n\geq1}\tau(w_n^*w_n b).
		\end{aligned}
	\end{equation*}
	Taking the infimum over $r$ and using the fact $2st=\inf_{r>0}\{r s^2+r^{-1}t^2\}$, we get
	\begin{equation*}
		\begin{aligned}
			\sum_{n\geq1}|\tau(z_n^*v_nw_n)|&\leq \Big(\sum_{n\geq1}\tau(v_nv_n^*a)\Big)^{\frac 1 2}\Big(\sum_{n\geq1}\tau(w_n^*w_n b)\Big)^{\frac 1 2}\\
			&=\Big(\sum_{n\geq 1}\|a^{\frac 1 2}v_n\|_2^2\Big)^{\frac 1 2}\Big(\sum_{n\geq 1}\|b^{\frac 1 2}w_n^*\|_2^2\Big)^{\frac 1 2}.
		\end{aligned}
	\end{equation*}
	In particular, the above implies that for any $v\in E_{1-\theta}(\mathcal M)$ and $w\in E_\theta(\mathcal M)$,
	\begin{equation*}
		|\tau(z_n^* vw)|\leq \|a^{\frac 1 2} v\|_2\|wb^{\frac 1 2}\|_2,\quad n\geq 1,
	\end{equation*}
	Applying Lemma \ref{key-lem}, we conclude that $z=(z_n)_{n\geq 1}\in E^\times(\mathcal M;\ell_\infty^\theta)$ and
	$$\|z\|_{E^\times(\mathcal M;\ell_\infty^\theta)}\leq 1.$$
	According to Remark \ref{dense-remark}, $\mathfrak{F}$ is dense in $E(\mathcal M;\ell_1^\theta)$. Therefore, the functional $\phi$ is uniquely determined by the sequence $z=(z_n)_{n\geq 1}$ and  the duality is proved.
\end{proof}


We also have the  duality between $\h_E^{1_\theta}$ and $\h_{E^\times}^{\infty_\theta}$.

\begin{cor}\label{duality}Let $0\leq\theta\leq 1$ and let $E$ be a symmetric Banach function space satisfying $E\in \emph{Int}(L_p,L_q)$ with $1<p\leq q<2$. If $E$ is $2$-concave, then
	$$(\h_E^{1_\theta}(\mathcal M))^*=\h_{E^\times}^{\infty_\theta}(\mathcal M)$$
	with equivalent norms.
\end{cor}
\begin{proof} Since $E$ is $2$-concave, it follows that $E$ is separable and $E^\times$ is $2$-convex (see e.g. \cite[Lemma 4.12, Theorem 4.13]{Dirksen-Thesis}). On the other hand, according to \cite[Page 52]{LT}, if $E$ is $q$-concave with $q<\infty$, then $E$ does not have a subspace which is isomorphic to $c_0$. Combining this fact with the separability of $E$, we conclude that $E$ has Fatou property (see e.g. \cite[Page 119]{LT}). Therefore, Theorem \ref{key-prop} and Proposition \ref{complemented-lem} are applicable here. A combination of Proposition \ref{complemented-lem} and  Theorem \ref{key-prop} yields the desired duality.
\end{proof}

Now we provide the proof of Theorem \ref{duality-ns}.

\begin{proof}[Proof of Theorem \ref{duality-ns}]
 Since $E\in \mbox {Int}(L_p,L_q)$ with $1<p\leq q<2$ and $E$ is $2$-concave, it follows from \cite[Theorem 3.2]{Bekjan2018} that
$$(\h_E^c)^*=\h_{E^\times}^c,\quad (\h_E^r)^*=\h_{E^\times}^r.$$
Combining this with Theorem \ref{discrete-version-E} and Corollary \ref{duality}, we complete the proof.
\end{proof}

\begin{rem} Indeed, it can be seen from above that the assumption of $E$ in Theorem \ref{duality-ns} can be relaxed to the following: $E$ is a separable symmetric Banach function space with Fatou property and $E^\times$ is $2$-convex.
\end{rem}

\medskip

\section{Comments for asymmetric Johnson-Schechtman inequalities}

We now turn to the asymmetric versions of  noncommutative Johnson-Schechtman inequalities. Let $L_0^h(\mathcal M)$ denote the set of all self-adjoint elements in $L_0(\mathcal M)$. We say that $x\in L_1(\mathcal M) \cap L_0^h(\mathcal M)$ is mean zero if $\tau(x)=0$. We introduce the definition of noncommutative independence in the sense of Junge and Xu \cite{JX2008}.

\begin{defi} Let $\mathcal M$ be a finite von Neumann algebra equipped with a finite faithful trace $\tau$. Assume that $(\mathcal M_k)_{k\geq 1}$ are von Neumann subalgebras of $\mathcal M$.
\begin{enumerate}[\rm (i)]
\item We say that $(\mathcal M_k)_{k\geq 1}$ are independent with respect to $\tau$ if $\tau(xy)=\tau(x)\tau(y)$ holds true for every $x\in\mathcal M_k$ and for every $y$ in the von Neumann algebra generated by $(\mathcal M_j)_{j\neq k}$.

\item A sequence $(x_k)_{k\geq 1}\subset L_0^h(\mathcal M)$ is said to be independent with respect to $\tau$ if the unital von Neumann subalgebras $\mathcal M_k$, $k\geq 1$, generated by $x_k$ are independent.
\end{enumerate}
\end{defi}

Similar to Theorem \ref{discrete-version-E}, we also have asymmetric form of Johnson-Schechtman inequalities for noncommutative independent random variables.

\begin{thm}\label{main-result-2} Let $0\leq \theta\leq 1$. Assume that $E$ is a symmetric Banach function space  which is an interpolation of the couple $ (L_p,L_q)$ for $1<p\leq q<\infty$. If $E$ is $2$-convex with Fatou norm, then for any sequence $(x_k)_{k\geq 1}$ of mean zero independent random variables,
\begin{equation}\label{asymmetric-js-inequality}
\Big\|\sum_{k\geq 1} x_k\Big\|_E\simeq_E \|(x_k)_{k\geq 1}\|_{E(\mathcal M;\ell_\infty^\theta)}+\Big\|\sum_{k\geq 1} x_k\otimes e_k\Big\|_{(L_1+L_2)(\mathcal M\overline\otimes \ell_\infty)}.
\end{equation}
\end{thm}

\begin{proof} The proof is similar to that of  Theorem \ref{discrete-version-E}.  We include details for the convenience of the reader. It follows from \cite[Theorem 1.5]{JZZ} that
$$\Big\|\sum_{k\geq 1} x_k\Big\|_E\lesssim_E \|(x_k)_{k\geq 1}\|_{E(\mathcal M;\ell_\infty)}+\Big\|\sum_{k\geq 1} x_k\otimes e_k\Big\|_{(L_1+L_2)(\mathcal M\overline\otimes \ell_\infty)}.$$
Clearly, if $\|(x_k)_{k\geq 1}\|_{E(\mathcal M;\ell_\infty)}\leq \|\sum_{k\geq 1} x_k\otimes e_k\|_{(L_1+L_2)(\mathcal M\overline\otimes \ell_\infty)}$, then we get the desired inequality immediately. Now suppose that
$$\Big\|\sum_{k\geq 1} x_k\otimes e_k\Big\|_{(L_1+L_2)(\mathcal M\overline\otimes \ell_\infty)}\leq \|(x_k)_{k\geq 1}\|_{E(\mathcal M;\ell_\infty)}.$$
Then we have
$$\Big\|\sum_{k\geq 1} x_k\Big\|_E\lesssim_E \|(x_k)_{k\geq 1}\|_{E(\mathcal M;\ell_\infty)}.$$
For $\frac 1 2<\theta\leq 1$, by Theorem \ref{complex-interpolation-result-1}, we get
$$\|(x_k)_{k\geq 1}\|_{E(\mathcal M; \ell_\infty)}\leq \|(x_k)_{k\geq 1}\|_{E(\mathcal M; \ell_\infty^{r})}^{1-\eta}\|(x_k)_{k\geq 1}\|_{E(\mathcal M; \ell_\infty^{\theta})}^{\eta},\quad \eta=\frac{\theta-1}{2\theta}.$$
By the definition of $E(\mathcal M; \ell_\infty^{r})$ and \cite[Theorem 1.4]{JSZ2016}, we have
$$\|(x_k)_{k\geq 1}\|_{E(\mathcal M; \ell_\infty^{r})} \leq  \Big\|\Big(\sum_{k\geq 1}|x_k^*|^2\Big)^{\frac 1 2}\Big\|_E\simeq_E \Big\|\sum_{k\geq 1} x_k\Big\|_E.$$
Combining the last three estimates, we conclude that
$$\Big\|\sum_{k\geq 1} x_k\Big\|_E\lesssim_E \|(x_k)_{k\geq 1}\|_{E(\mathcal M;\ell_\infty^\theta)}.$$
The case $0\leq\theta<\frac 1 2$ can be dealt with in a similar manner. Therefore, we obtain
$$\Big\|\sum_{k\geq 1} x_k\Big\|_E\lesssim_E \|(x_k)_{k\geq 1}\|_{E(\mathcal M;\ell_\infty^\theta)}+\Big\|\sum_{k\geq 1} x_k\otimes e_k\Big\|_{(L_1+L_2)(\mathcal M\overline\otimes \ell_\infty)}.$$

For the converse inequality, the proof is similar to Theorem \ref{discrete-version-E}. We leave details to the reader.
\end{proof}

As a consequence of Theorem \ref{main-result-2}, if we consider noncommutative positive independent random variables, we will have the following corollary. The proof is routine, and therefore we omit the details.

\begin{cor}\label{main-corollary-2} Let $0\leq \theta\leq 1$. Assume that $E$ is a symmetric Banach function space  which is an interpolation of the couple $ (L_p,L_q)$ for $1<p\leq q<\infty$.  If $E$ is $2$-convex with Fatou norm, then for any sequence $(x_k)_{k\geq 1}$ of positive independent random variables,
\begin{equation*}\label{asymmetric-js-inequality-cor}
\Big\|\sum_{k\geq 1} x_k\Big\|_E\simeq_E \|(x_k)_{k\geq 1}\|_{E(\mathcal M;\ell_\infty^\theta)}+\Big\|\sum_{k\geq 1} x_k\otimes e_k\Big\|_{L_1(\mathcal M\overline\otimes \ell_\infty)}.
\end{equation*}
\end{cor}

\noindent {\bf Acknowledgements.}
Authors are grateful to Professor Yong Jiao, Professor Narcisse Randrianantoanina and Professor Eric Ricard for some helpful comments. Authors are grateful to their home institutions (CSU, UofG) for never ending support of their research.

%
%
%
%
%


\end{document}